\documentclass{asl}

\usepackage[urlbordercolor={0 0 .8}]{hyperref}
\usepackage{mathrsfs}
\usepackage{longtable}
\usepackage{enumerate}
\usepackage{proof} %for proofs
\theoremstyle{plain}
\newtheorem{prop}{Proposition}[section]
\newtheorem{lemm}[prop]{Lemma}
\newtheorem{theo}[prop]{Theorem}
\newtheorem{coro}[prop]{Corollary}
\newtheorem{folk}[prop]{Folklore}

\theoremstyle{definition}
\newtheorem{definition}[prop]{Definition}

\newcommand{\Ord}{\mathsf{Ord}}
\newcommand{\LOrd}{\mathsf{LOrd}}
\newcommand{\COrd}{\mathsf{COrd}}

\newcommand{\UCF}{\mathtt{U}}

\newcommand{\Ns}{\mathscr{N}}
\newcommand{\ncplus}{\boxplus}

\newcommand{\IAL}{\mathbf{ACT}_{\omega}}
\newcommand{\EIAL}{{!} \IAL}
\newcommand{\EIALminus}{\EIAL^{-}}
\newcommand{\EIALm}{\EIAL^{\mathrm{m}}}
\newcommand{\EIALne}{\EIAL^{\mathrm{ne}}}
\newcommand{\GEIALminus}{{!} \underline{\mathbf{ACT}}_{\omega}^-}
\newcommand{\dEIALm}{{!}_{\mathrm{d}} \IAL^{\mathrm{m}}}
\newcommand{\dEIALne}{{!}_{\mathrm{d}} \IAL^{\mathrm{ne}}}
\newcommand{\GdEIALne}{{!}_{\mathrm{d}} \underline{\mathbf{ACT}}_\omega^{\mathrm{ne}}}

\newcommand{\BS}{\mathbin{\backslash}}
\newcommand{\SL}{\mathbin{/}}
\newcommand{\mconj}{\cdot}
\newcommand{\aconj}{\mathbin{\&}}
\newcommand{\adisj}{\oplus}
\newcommand{\yields}{\Rightarrow}
\newcommand{\U}{1}
\newcommand{\Z}{0}

\newcommand{\Hc}{\mathcal{H}}
\newcommand{\KA}{\mathbf{K}}
\newcommand{\AL}{\mathbf{A}}
\newcommand{\Mf}{\mathcal{M}}

\newcommand{\ID}{(\mathrm{Id})}
\newcommand{\CUT}{(\mathrm{Cut})}

\newcommand{\Var}{\mathrm{Var}}

\newcommand{\dom}{\mathop{\mathrm{dom}}}

\DeclareMathOperator*{\bigdoublewedge}{\bigwedge\mkern-18mu\bigwedge}
\DeclareMathOperator*{\bigdoublevee}{\bigvee\mkern-18mu\bigvee}

\newcommand{\VARlt}{\mathit{lt}}
\newcommand{\VARrt}{\mathit{rt}}
\newcommand{\VARok}{\mathit{ok}}
\newcommand{\VARen}{\mathit{en}}
\newcommand{\VARfn}{\mathit{fn}}
\newcommand{\VARgo}{\mathit{go}}
\newcommand{\VARex}{\mathit{ex}}
\newcommand{\VARfl}{fl}
\newcommand{\VARSigma}{p_{\exists}}
\newcommand{\VARPi}{p_{\forall}}
\newcommand{\VARsym}{p}

\newcommand{\FOROk}{\mathrm{Ok}}
\newcommand{\FORBrk}{\mathrm{Brk}}
\newcommand{\FOREn}{\mathrm{En}}
\newcommand{\FORCmp}{\mathrm{Cmp}}

\newcommand{\TM}{\mathtt{M}}
\newcommand{\qbgn}{q^{(0)}}
\newcommand{\qacc}{q^{(a)}}

\newcommand{\SR}{\mathtt{S}}

\newcommand{\fun}{\mathfrak{f}}

\newcommand{\bsc}{\mathrm{bsc}}
\newcommand{\IALbsc}{\IAL^{\bsc}(\SR)}
\newcommand{\SFm}{\mathrm{SFm}}

\newcommand{\ips}{\mathrm{ips}}

\begin{document}

\title[Reasoning from hypotheses in $\ast$-continuous action lattices]{Reasoning from hypotheses\\ in $\ast$-continuous action lattices}

\author[S.\,L.\ Kuznetsov]{Stepan L.\ Kuznetsov}
\revauthor{Kuznetsov, Stepan L.}

\author[T.\ Pshenitsyn]{Tikhon Pshenitsyn}
\revauthor{Pshenitsyn, Tikhon}

\author[S.\,O.\ Speranski]{Stanislav O.\ Speranski}
\revauthor{Speranski, Stanislav O.}

\address{Steklov Mathematical Institute\\ of Russian Academy of Sciences,\\ Moscow, Russia}

\email{sk@mi-ras.ru, tpshenitsyn@mi-ras.ru, katze.tail@gmail.com}

\date{\today}

%%%

\begin{abstract}
The class of all $\ast$-continuous Kleene algebras, whose description includes an infinitary condition on the iteration operator, plays an important role in computer science. The complexity of reasoning in such algebras --- ranging from the equational theory to the Horn one, with restricted fragments of the latter in between --- was analyzed by Kozen (2002). This paper deals with similar problems for $\ast$-continuous residuated Kleene lattices, also called $\ast$-con\-ti\-nu\-o\-us action lattices, where the product operation is augmented by  residuals. We prove that, in the presence of residuals, the fragment of the corresponding Horn theory with $\ast$-free hypotheses has the same complexity as the $\omega^\omega$ iteration of the halting problem, and hence is properly hyperarithmetical. We also prove that if only commutativity conditions are allowed as hypotheses, then the complexity drops down to $\Pi^0_1$ (i.e.\ the complement of the halting problem), which is the same as that for $\ast$-continuous Kleene algebras. In fact, we get stronger upper bound results: the fragments under consideration are translated into suitable fragments of infinitary action logic with exponentiation, and our upper bounds are obtained for the latter ones.
\end{abstract}



\maketitle

%%%

\section{Introduction}

The notion of Kleene algebra is one of the fundamental notions in computer science. It arose in various forms in formal language theory, automata theory, program semantics, relational algebra, etc. The formal definition was given by Kozen~\cite{Kozen1994}. Kleene algebras and their variations are widely used in reasoning about program correctness, as an efficient tool for specification and verification. 

The language of Kleene algebras includes two binary operations: $\mconj$ (product) and $\adisj$ (join, or sum), two constants: $\Z$ (zero) and $\U$ (unit), and, most importantly, the unary operation ${}^{\ast}$ of iteration, also called \emph{Kleene star}. The latter goes back to the seminal work of Kleene~\cite{Kleene1956}.

\begin{definition}\label{Df:KA}
A {\em Kleene algebra} is a structure $\KA = (K; \mconj, \adisj, {}^*, 0, 1)$
such that:
\begin{itemize}
\item $(K; \adisj, \mconj, 0, 1)$ is a semiring;
\item $a \adisj a = a$ for all $a \in K$ (idempotency);
\item for any $a,b \in K$,
\begin{align*}
    & 1 \adisj a \mconj a^* \leqslant a^* ; && 1 \adisj a^* \mconj a \leqslant a^*; \\
    & \mbox{if } a \mconj b \leqslant b,\mbox{ then } a^* \mconj b \leqslant b;  &&
    \mbox{if } b \mconj a \leqslant b,\mbox{ then } b \mconj a^* \leqslant b.
\end{align*}
\end{itemize}
Here $\leqslant$ denotes the partial ordering of $K$ given by: $a \leqslant b$ if{f} $a \adisj b = b$.
\end{definition}

Obviously, the class of all Kleene algebras is finitely axiomatizable in first-or\-der logic, and hence its first-order theory is computably enumerable --- or, in other words, $\Sigma^0_1$.

The {\em equational} theory of Kleene algebras, i.e.\ the set of all equations $A = B$ (where $A$ and $B$ are terms) that are true in all Kleene algebras, is known to be algorithmically decidable. More precisely, two terms are equivalent in all Kleene algebras if{f} the corresponding regular expressions describe the same language~\cite{Kozen1994}, and the latter problem is well-known to be $\mathsf{PSPACE}$-comp\-lete.

However, already a very simple fragment of the Horn theory of Kleene algebras turns out to be undecidable. Namely, as was proved in \cite{Kuzn2023ICTAC}, there exists a finite set $S$ of four {commutativity  conditions} such that the set of equations entailed by $S$  is $\Sigma^0_1$-complete, i.e., is equivalent to the halting problem.\footnote{By
\emph{commutativity conditions} we mean equations of the form $x \cdot y = y \cdot x$ where $x$ and $y$ are variables.}
Independently, a similar result was obtained in \cite{AzevedoAmorim2024}. Therefore, the Horn theory and the first-order theory of Kleene algebras are both $\Sigma^0_1$-complete.

A much more interesting complexity landscape opens up for a narrower class of Kleene~algebras, which is characterized by a natural infinitary condition:

\begin{definition}\label{Df:contKA}
A Kleene algebra is {\em $*$-con\-ti\-nu\-ous} if{f} for any $a, b, c \in K$ we have
\[
{b \cdot a^* \cdot c}\ =\
{\sup_{\leqslant} \{ b \cdot a^n \cdot c \mid n \geqslant 0 \}} .
\]
\end{definition}

The class of all $\ast$-continuous Kleene algebras, unlike that of all Kleene algebras, cannot be axiomatized in first-order logic. So its first-order theory may be far from $\Sigma^0_1$.
Still, as shown in \cite{Kozen1994}, the equational theory of $*$-con\-ti\-nu\-ous Kleene algebras coincides with that of arbitrary Kleene algebras, and hence is decidable. As for Horn theories (entailment from hypotheses), we have the following complexity results from \cite{Kozen2002}:

\begin{center}
\renewcommand{\arraystretch}{1.2}
\begin{tabular}{|c|c|} \hline
{\bfseries Hypotheses} & {\bfseries Complexity} \\ \hline\hline
none (equational theory) & $\mathsf{PSPACE}$-comp\-lete \\ \hline
commutativity conditions ($x \mconj y = y \mconj x$) & \vphantom{\Large A} $\Pi^0_1$-complete \\ \hline
$*$-free ($A = B$, where $A$ and $B$ do not include ${}^*$) & \vphantom{\Large A} $\Pi^0_2$-complete \\\hline
arbitrary & \vphantom{\Large A} $\Pi^1_1$-complete \\\hline
\end{tabular}
\end{center}
Note that the results stated in the last two lines also hold for commutative $*$-con\-ti\-nu\-ous Kleene algebras; see \cite{Kuzn2024Izv}.

In this article, we consider theories of {residuated Kleene lattices}, also called \emph{action lattices}. These structures arise in the works of Pratt~\cite{Pratt1991} and Kozen~\cite{Kozen1994act}. Briefly, they are obtained from Kleene algebras by adding meet and two residuals for the product operation.

\begin{definition}\label{Df:action}
An \emph{action lattice}, or a \emph{residuated Kleene lattice}, is a structure $\AL = (A; \mconj, \adisj, \aconj, {\BS}, {\SL}, {}^*, \Z, \U)$ such that:
\begin{enumerate}
\item $(A; \cdot, \adisj, {}^*, \Z, \U)$ is a Kleene algebra;
\item $(A; \adisj, \aconj)$ is a lattice (where $\adisj$ is join and $\aconj$ is meet); 
\item $\BS$ and $\SL$ are residuals of the product operation w.r.t.\ the partial ordering:
\[
{b \leqslant a \BS c}\ \iff\ {a \mconj b \leqslant c}\ \iff\ {a \leqslant c \SL b} .
\]
\end{enumerate}
\end{definition}

It should be noted that residuals are division operations associated with the partial ordering rather than equality. The idea of residuation in algebra and logic goes back to the works of Krull \cite{Krull1924}, Ward and Dilworth \cite{WardDilworth1939}, and Lambek \cite{Lambek1958}. For more information about the connections between residuated algebraic structures and substructural logics, one may consult \cite{Galatos}.

The class of all action lattices is finitely axiomatizable in first-order, and hence its first-order theory belongs to $\Sigma^0_1$. This bound is precise, since the equational theory of action lattices is $\Sigma^0_1$-hard \cite{Kuzn2021TOCL}.

\begin{definition}\label{Df:action_cont}
An action lattice $\AL = (A; \mconj, \adisj, \aconj, \BS, \SL, {}^*, \Z, \U)$ is
\emph{$*$-con\-ti\-nu\-ous} if{f} so is the underlying Kleene algebra $\KA = (A; \mconj, \adisj, {}^*, \Z, \U)$.
\end{definition}

It is known that the equational theory of $*$-con\-ti\-nu\-ous action lattices is $\Pi^0_1$-comp\-lete, i.e., is equivalent to the complement of the halting problem; see \cite{Buszkowski2007} and \cite{Palka2007}. On the other hand, the corresponding Horn theory, with no restrictions on hypotheses, is $\Pi^1_1$-comp\-lete: the lower bound here is inherited from $*$-con\-ti\-nu\-ous Kleene algebras, and the upper one is established by a very general infinitary proof-theoretic argument presented in~\cite{KuznetsovSperanski2022}.

We are going to focus on restricted fragments of the Horn theory of $*$-con\-ti\-nu\-ous action lattices: the fragment with $*$-free hypotheses and its smaller subfragments in which hypotheses are required to be \emph{monoidal inequations}. Recall that the latter are expressions of the form
\[
{b_1 \mconj \ldots \mconj b_n}\ \leqslant\ {c_1 \mconj \ldots \mconj c_m},
\]
where the $b_i$'s and the $c_j$'s are variables.\footnote{Here
we assume that the empty product is identified with $1$.}
A \emph{non-expanding monoidal inequation} is one where $m \leqslant n$; we also call them \emph{non-expanding hypotheses}. Obviously, any commutativity condition $b \mconj c = c \mconj b$ can be represented as two non-expanding hypotheses, viz.\
\[
{b \mconj c} \leqslant {c \mconj b}
\quad \text{and} \quad
{c \mconj b} \leqslant {b \mconj c} .
\]
We shall prove that reasoning from non-expanding hypotheses --- and, in particular, from commutativity conditions --- for $\ast$-con\-ti\-nu\-ous action lattices has the same complexity as in the case of $\ast$-con\-ti\-nu\-ous Kleene algebras, i.e., is $\Pi^0_1$-complete. The lower bound here is inherited either from the equational theory of $\ast$-con\-ti\-nu\-ous action lattices or from reasoning from com\-mu\-ta\-ti\-vi\-ty conditions for $\ast$-con\-ti\-nu\-ous Kleene algebras (see above). However, the upper bound result here is new. In sharp contrast, if we allow arbitrary monoidal inequations, the complexity becomes hyperarithmetical, as we shall see below.

In fact, we obtain the $\Pi^0_1$-bounded\-ness for a richer system. To be more precise, let us extend {\em infinitary action logic}, i.e.\ the logic of $\ast$-con\-ti\-nu\-ous action lattices, denoted by $\IAL$, with the exponential modality ${!}$ from Girard's linear logic (see \cite{Girard1987}). This leads to {\em infinitary action logic with exponentiation}, introduced in \cite{KuznetsovSperanski2022} and denoted by $\EIAL$. Next, ${!}$ gives us a modalised version of deduction theorem, which  reduces reasoning from hypotheses in $\IAL$ to deriving sequents in $\EIAL$. Then we show the $\Pi^0_1$-boun\-ded\-ness of the restricted derivability problem for $\EIAL$, where ${!}$ can be applied only to formulae of the form
\[
(b_1 \mconj \ldots \mconj b_n) \BS (c_1 \mconj \ldots \mconj c_m)
\]
with $m \leqslant n$ --- naturally, they correspond to non-expanding hypotheses. A somewhat similar result was obtained earlier in~\cite{Kuzn2020ScFest}, where the ${!}$-for\-mu\-lae were required to be of the same form, but with the condition $m \leqslant n$ replaced by $m = 1$. However, this does~not suffice for our purposes, since commutativity conditions are not covered by such ${!}$-for\-mu\-lae.

Further, reasoning from $*$-free hypotheses in $\ast$-con\-ti\-nu\-ous action lattices corresponds to the fragment of $\EIAL$ where ${}^{\ast}$ may~not occur in the scope of ${!}$. It was proved earlier in \cite{Kuzn2021Tableaux} that this fragment is $\Pi^0_2$-hard and $\Delta^1_1$-bounded. We are going to obtain much stronger results, which allow us to pinpoint the exact complexity bound. More precisely, we prove that the derivability problem for the fragment in question is equivalent to the $\omega^\omega$ transfinite iteration of the halting problem, and hence is properly hyperarithmetical. In what follows, we shall refer to this level of complexity as \emph{$\Sigma^0_{\omega^\omega}$-comp\-lete\-ness}. Our usage here is not entirely standard, since the limit levels of the hyperarithmetical hierarchy are often defined using an auxiliary application of the Turing jump operator, which is related to the properties of Kleene's system of notation for constructive ordinals. As far as only ordinals below $\omega^\omega$ are concerned, a much simpler system of notation can be utilised, which leads to a more natural definition of the limit levels. The details are given in Section~\ref{subsec-ha}.

To be more precise, we prove the $\Sigma^0_{\omega^{\omega}}$-bounded\-ness for the fragment of $\EIAL$ where ${!}$ is applied only to $\ast$-free formulae, while the $\Sigma^0_{\omega^{\omega}}$-hard\-ness is obtained for reasoning from mo\-no\-i\-dal equations in $\ast$-con\-ti\-nu\-ous action lattices. This yields the same exact complexity bound for stronger theories: reasoning from arbitrary $\ast$-free hypotheses and deriving sequents in the corresponding fragment of $\EIAL$.

To sum up, the situation with $\ast$-con\-ti\-nu\-ous action lattices is significantly different from that with $\ast$-con\-ti\-nu\-ous Kleene algebras, where reasoning from $\ast$-free hypotheses is $\Pi^0_2$-comp\-le\-te in terms of the arithmetical hierarchy (see \cite{Kozen2002}). Moreover, as was proved in \cite{KuznetsovSperanski2022}, the derivability problem for $\EIAL$ without any restrictions is $\Pi^1_1$-comp\-lete --- like for the full Horn theory of $\ast$-con\-ti\-nu\-ous action lattices. So here is the residuated version of Kozen's table:
%\begin{center}
\renewcommand{\arraystretch}{1.2}
\begin{longtable}{|c|c|} \hline
{\bfseries Hypotheses} & {\bfseries Complexity} \\ \hline\hline
none (equational theory) &  $\Pi^0_1$-complete \\ \hline
non-expanding hypotheses, &  $\Pi^0_1$-complete \\[-2pt]
including commutativity conditions & \\ \hline
monoidal inequations & $\Sigma^0_{\omega^\omega}$-complete \\ \hline
$\ast$-free ($A = B$ where $A$ and $B$ do~not contain ${}^{\ast}$) & \vphantom{\Large A} $\Sigma^0_{\omega^\omega}$-complete \\\hline
arbitrary & $\Pi^1_1$-complete \\\hline
\end{longtable}
%\end{center}
In addition, the same complexity results hold for the corresponding fragments of $\EIAL$.

The rest of the paper is organised as follows. Section~\ref{S:prelim} contains some preliminary material: in Section~\ref{S:IAL}, we describe infinitary action logic $\IAL$ and its exponential extension $\EIAL$; Section~\ref{S:dyadic} presents an auxiliary dyadic calculus, which turns out to be useful for a systematic analysis of derivations; in Section~\ref{S:inductive}, we give necessary information on inductive de\-fi\-ni\-tions, which is needed for ordinal analysis of derivability  operators that arise in our study; Section~\ref{subsec-ha} describes the hyperarithmetical hierarchy up to $\omega^\omega$. In Section~\ref{sec-dyadic}, we prove some properties of the dyadic system introduced in Section~\ref{S:dyadic}.

The main technical results are presented in Sections~\ref{S:Tikhon}, \ref{S:Stanislav}, and~\ref{S:Stepan}. In Section~\ref{S:Tikhon} we show the $\Sigma^0_{\omega^\omega}$-hard\-ness of reasoning from monoidal equations in $\ast$-con\-ti\-nu\-ous action lattices. Next, in Section~\ref{S:Stanislav}, we prove that the restricted derivability problem for $\EIAL$ where ${!}$ can be applied only to $\ast$-free formulae is $\Sigma^0_{\omega^\omega}$-boun\-ded. So $\Sigma^0_{\omega^\omega}$ is indeed the exact complexity bound. This also leads to some natural results about closure ordinals. Finally, in Section~\ref{S:Stepan}, we establish the $\Pi^0_1$-boun\-ded\-ness of deriving sequents in the fragment of $\EIAL$ where ${!}$ can be applied only to formulae that represent non-expanding hypotheses. 

%The article is also equipped with an Appendix, which accurately presents the $\Sigma^0_{\omega^\omega}$-complete problem the reduction to which is used in Section~\ref{S:Tikhon} for proving the hyperarithmetical lower bound --- namely, the validity problem for computable infinitary propositional formulae of rank less than $\omega^\omega$.

%%%

\section{Preliminaries}\label{S:prelim}

\subsection{Infinitary action logic with exponentiation}\label{S:IAL}

Let us recall infinitary action logic $\IAL$ as formulated by Palka~\cite{Palka2007}. Formulae of $\IAL$, which are actually terms in the language of action lattices, are built from a countable set of variables $\Var$ and constants $\Z$ and $\U$ using five binary connectives: $\BS$, $\SL$, $\mconj$, $\adisj$, $\aconj$, and one unary connective ${}^*$ (written in the postfix form: $A^*$). Formulae are denoted by letters $A,B,C$, possibly with subscripts.  Greek letters $\Gamma, \Delta, \Pi, \Xi, \Upsilon$ (maybe with subscripts) are used for finite sequences of formulae (including the empty one).
{\em Sequents} are expressions of the form $\Pi \yields B$, where $B$ is a formula and $\Pi$ is a sequence of formulae. The rules of $\IAL$ are as follows.
\[
\infer[\ID]
{A \yields A}
{}
\]
\[
\infer[({\BS} L)]{\Gamma, \Pi, A \BS B, \Delta \yields C}
{\Pi \yields A & \Gamma, B, \Delta \yields C}
\qquad
\infer[({\BS} R)]{\Pi  \yields A \BS B}
{A, \Pi \yields B}
\]
\[
\infer[({\SL} L)]{\Gamma, B \SL A, \Pi, \Delta \yields C}
{\Pi \yields A & \Gamma, B, \Delta \yields C}
\qquad
\infer[({\SL} R)]{\Pi \yields B \SL A}
{\Pi, A \yields B}
\]
\[
\infer[(\mconj L)]{\Gamma, A \mconj B, \Delta \yields C}
{\Gamma, A, B, \Delta \yields C}
\qquad
\infer[(\mconj R)]{\Gamma, \Delta \yields A \mconj B}
{\Gamma \yields A & \Delta \yields B}
\]
\[
\infer[(\U L)]{\Gamma, \U, \Delta \yields C}{\Gamma, \Delta \yields C}
\qquad
\infer[(\U R)]{\yields\U}{}
\qquad
\infer[(\Z L)]{\Gamma,\Z,\Delta \yields C}{}
\]
\[
\infer[(\adisj L)]{\Gamma, A_1 \adisj A_2, \Delta \yields C}
{\Gamma, A_1, \Delta \yields C & \Gamma, A_2, \Delta \yields C}
\qquad
\infer[(\adisj R_i)\mbox{, $i = 1,2$}]
{\Pi \yields A_1 \adisj A_2}{\Pi \yields A_i}
\]
\[
\infer[({\aconj} L_i)\mbox{, $i = 1,2$}]
{\Gamma, A_1 \aconj A_2, \Delta \yields C}
{\Gamma, A_i, \Delta \yields C}
\qquad
\infer[({\aconj} R)]{\Pi \yields A_1 \aconj A_2}
{\Pi \yields A_1 & \Pi \yields A_2}
\]
\[
\infer[({}^* L_\omega)]
{\Gamma, A^*, \Delta \yields C}
{\bigl( \Gamma, A^n, \Delta \yields C \bigr)_{n \in \omega}}
\qquad
\infer[({}^* R_n),\ n \ge 0] %,\mbox{ each $\Pi_i$ is non-empty}]
{\Pi_1, \ldots, \Pi_n \yields A^*}
{\Pi_1 \yields A & \ldots & \Pi_n \yields A}
\]
\[
\infer[\CUT]
{\Gamma, \Pi, \Delta \yields C}
{\Pi \yields A & \Gamma, A, \Delta \yields C}
\]

\noindent

Rules without premises, i.e.\ $(\mathrm{Id})$, $(\U R)$, $(\Z L)$, and $({}^* R_0)$, are \emph{axioms.} Their conclusions can be derived without assuming anything.

It should also be noted that (${}^* L_{\omega}$) is the only rule with infinitely many premises, i.e.\ the only $\omega$-rule.  In the presence of an $\omega$-rule, a derivation is defined as an infinite, but \emph{well-founded} tree, whose vertices are labelled by sequents, such that leaves are labelled with axioms or hypotheses (if we derive from hypotheses) and the labels of inner vertices can be obtained from labels of their immediate children by rule applications. There is a well-known alternative definition of the set of derivable sequents~\cite[Proposition~2.2]{KuznetsovSperanski2022}: the set of derivable sequents coincides with the least (w.r.t.\ inclusion) set of sequents including all hypotheses (if any) and axioms and closed under the inference rules. For more information on infinitary derivations, see~\cite{Aczel,Buchholz,Pohlers}. 

Infinitary action logic enjoys a natural algebraic interpretation on $*$-con\-ti\-nu\-ous action lattices.

\begin{definition}
    An algebraic model $\Mf$ of $\IAL$ is a pair $(\AL, w)$ where $\AL$ is an action lattice, and $w$ is an interpretation function mapping formulae of $\IAL$ to elements of $A$. The interpretation function is defined arbitrarily on variables, maps constants $0$ and $1$ to the corresponding designated elements of $A$, and commutes with operations. A sequent $A_1, \ldots, A_n \yields B$ is true in a model $\Mf$ (notation: $\Mf \vDash A_1, \ldots, A_n \yields B$) if
    $w(A_1) \cdot \ldots \cdot w(A_n) \leqslant w(B)$ in $\AL$. For $n = 0$ the truth condition is a bit different: $\Mf \vDash {} \yields B$, if $1 \leqslant w(B)$ in $\AL$.
\end{definition}

\begin{definition}
Let $S$ be a set of sequents (possibly an infinite one) and let $\Pi \yields C$ be a sequent. The sequent $\Pi \yields C$ is \emph{semantically entailed} by $S$ on the class of all $*$-con\-ti\-nu\-ous action lattices if for any model $\Mf$ in which all sequents from $S$ are true, so is $\Pi \yields C$.
Semantic entailment by $\varnothing$ is called \emph{general validity.}
\end{definition}

The following strong soundness-and-completeness theorem holds. (The proof is the same as given by Palka~\cite{Palka2007} for the weak version of completeness.)

\begin{theo}\label{theo-strong-completeness}
    A sequent $\Pi \yields C$ is derivable in $\IAL$ from a set of sequents $S$ if and only if $\Pi \yields C$ is semantically entailed by $S$ on the class of all $*$-con\-ti\-nuous action lattices.
\end{theo}

Now let us introduce the exponential extension of $\IAL$, denoted by $\EIAL$ and called \emph{infinitary action logic with exponentiation}~\cite{KuznetsovSperanski2022}.
The language of $\IAL$ is extended by a unary operation ${!}$, called the exponential modality. (Unlike~\cite{KuznetsovSperanski2022}, we use only one modality, not a family of \emph{subexponential} ones. For our needs, it is sufficient.) The exponential modality is written in the prefix form: ${!}A$, while Kleene star is postfix: $A^*$.

Formulae of the form ${!}A$ are called \emph{${!}$-for\-mu\-lae.} For a sequence of formulae $\Xi = A_1, \ldots, A_n$, let
${!}\Xi$ denote ${!}A_1, \ldots, {!}A_n$. 

For ${!}$-for\-mu\-lae, besides logical rules which introduce the exponential modality, there are also \emph{structural} rules of permutation, weakening, and contraction. To be more precise, the complete set of exponential rules for $\EIAL$ is as follows.
\[
\infer[({!}L)]
{\Gamma, {!}A, \Delta \yields C}
{\Gamma, A, \Delta \yields C}
\qquad
\infer[({!}R)]
{{!}\Pi \yields {!}B}
{{!}\Pi \yields B}
\]
\[
\infer[({!}P_1)]
{\Gamma, {!}A, \Pi, \Delta \yields C}
{\Gamma, \Pi, {!}A, \Delta \yields C}
\qquad
\infer[({!}P_2)]
{\Gamma, \Pi, {!}A, \Delta \yields C}
{\Gamma, {!}A, \Pi, \Delta \yields C}
\]
\[
\infer[({!}W)]
{\Gamma, {!}A, \Delta \yields C}
{\Gamma, \Delta \yields C}
\qquad
\infer[({!}C)]
{\Gamma, {!}A, \Delta \yields C}
{\Gamma, {!}A, {!}A, \Delta \yields C}
\]

Since for ${!}$-for\-mu\-lae we have permutation and contraction, in sequences of the form ${!}\Xi$ the order and multiplicity of formulae does not matter. In the view of this, we shall use the same notation ${!}\Xi$ also for the corresponding \emph{set} of formulae, 
$\{ {!}A_1, \ldots, {!}A_n \}$.

The calculus $\EIAL$ enjoys cut elimination:

\begin{theo}[{see \cite[Theorem 4.2]{KuznetsovSperanski2022}}] \label{Th:EACTomega-cutelim}
Any sequent derivable in $\EIAL$ can be derived without using $(\mathrm{Cut})$.
\end{theo}

For convenience, in what follows we shall by default assume that all $\EIAL$ derivations are cut-free.
A standard corollary of the cut elimination theorem is the invertibility of the rules $(\mconj L)$, $(\adisj L)$ and $({}^* L_\omega)$.
\begin{coro}\label{coro:invertibility}
The following rules are admissible in $\EIAL$:
\[
\infer[({}^* L)^{\mathrm{inv}},\ n \in \omega]
{\Gamma, A^n, \Delta \yields C}
{\Gamma, A^*, \Delta \yields C}
\qquad
\infer[(\mconj L)^{\mathrm{inv}}]
{\Gamma, A, B, \Delta \yields C}
{\Gamma, A \mconj B, \Delta \yields C}
\]
\[
\infer[(\adisj L)^{\mathrm{inv}},\ i = 1,2]
{\Gamma, A_i, \Delta \yields C}
{\Gamma, A_1 \adisj A_2, \Delta \yields C}
\]
\end{coro}

The interesting feature of ${!}$ is that it enables a modalised form of deduction theorem, that is, incorporating finite sets of hypotheses into the sequent being derived. ``Pure'' derivability of sequents, in  turn, is easier to analyse, due to cut elimination.

Let $S$ be a finite set of sequents in the language of $\IAL$. For each sequent of the form $A_1, \ldots, A_n \yields B$ in $S$, consider the formula $(A_1 \mconj \ldots \mconj A_n) \BS B$; if $n = 0$, then the corresponding formula is just $B$. Let $\Upsilon_{S}$ be a sequence of all such formulae, in an arbitrary order. 

\begin{theo}\label{Th:deduction}
A sequent $\Pi \yields C$ is derivable in $\IAL$ from a finite set of sequents $S$ if and only if the sequent ${!}\Upsilon_{S}, \Pi \yields C$ is derivable in $\EIAL$.
\end{theo}

This theorem is similar to \cite[Lemma~5.3]{KuznetsovSperanski2022}, \cite[Lemma~10]{KanKuzNigSce2019MSCS}, and other statements of similar form which can be found in literature. In order to make this article self-contained, we provide a complete proof.

\begin{proof}
For the ``only if'' direction, consider the derivation of $\Pi \yields C$ from $S$ in $\IAL$ and prepend the antecedent of each sequent with ${!}\Upsilon_{S}$. Inference rules  (including Cut) remain valid, possibly after adding $({!}P_1)$, $({!}P_2)$, and $({!}C)$ to handle the newly added ${!}$-for\-mu\-lae. Each axiom $\ID$ transforms into ${!}\Upsilon_S, A \yields A$, which is derived from $A \yields A$ by several applications of $({!}W)$. Finally, each sequent from $S$ translates into a derivable sequent (here and further the double line means several applications of a rule):
\[
\infer=[({!}W)]
{{!}\Upsilon_{S}, A_1, \ldots, A_n \yields B}
{\infer[({!}P_1)]{{!}((A_1 \mconj \ldots \mconj A_n) \BS B), A_1, \ldots, A_n \yields B}
{\infer[({!}L)]{A_1, \ldots, A_n, {!}((A_1 \mconj \ldots \mconj A_n) \BS B) \yields B}
{\infer[({\BS}L)]{A_1, \ldots, A_n, (A_1 \mconj \ldots \mconj A_n) \BS B \yields B}
{\infer=[({\mconj}R)]{A_1, \ldots, A_n \yields A_1 \mconj \ldots \mconj A_n}
{A_1 \yields A_1 & \ldots & A_n \yields A_n} & B \yields B}}}
}\]
For $n=0$, the derivation is simpler:
\[
\infer=[({!}W)]
{{!}\Upsilon_{S} \yields B}
{\infer[({!}L)]{{!}B \yields B}{B \yields B}}
\]
Thus, we obtain a valid derivation of ${!}\Upsilon_{S}, \Pi \yields C$ in $\EIAL$ (without any hypotheses).

For the ``if'' direction, we assume  (Theorem~\ref{Th:EACTomega-cutelim}) that we have a cut-free derivation of ${!}\Upsilon_{S}, \Pi \yields C$ in $\EIAL$. Next, we erase all ${!}$-for\-mu\-lae from this derivation. Since the derivation is cut-free, such formulae are exactly formulae from ${!}\Upsilon_{S}$. Structural rules operating ${!}$-for\-mu\-lae, namely, $({!}P_1)$, $({!}P_2)$, $({!}W)$, and $({!}C)$, trivialise. All rules not operating ${!}$-for\-mu\-lae remain valid. The $({!}R)$ rule may never be applied, since ${!}$-for\-mu\-lae do not appear in right-hand sides of sequents. Finally, each application of the $({!}L)$ rule transforms to the following:
\[
\infer
{\Gamma, \Delta \yields C}
{\Gamma, (A_1 \mconj \ldots \mconj A_n) \BS B, \Delta \yields C}
\]
Here $(A_1, \ldots, A_n \yields B) \in S$, so this application is simulated using $\CUT$ as follows:
\[
\infer[\CUT]
{\Gamma, \Delta \yields C}
{\infer[({\BS}L)]{{} \yields (A_1 \mconj \ldots \mconj A_n) \BS B}
{\infer=[({\mconj}L)]{A_1 \mconj \ldots \mconj A_n \yields B}{\overbrace{A_1, \ldots, A_n \yields B}^{\in S}}}
& \Gamma, (A_1 \mconj \ldots \mconj A_n) \BS B, \Delta \yields C}
\]
(For $n=0$ we have just a cut with the sequent ${}\yields B$ from $S$.)
\end{proof}

Theorem~\ref{Th:deduction} establishes  one-way connections (embeddings) between reasoning from finite sets of hypotheses of specific form and fragments of $\EIAL$ with corresponding restrictions on ${!}$-for\-mu\-lae. In the table below we also introduce notations for these fragments. (The third notation comes from~\cite{Kuzn2021Tableaux}, the first two are new.)

\begin{center}
\begin{tabular}{|c|c|c|}\hline
Hypotheses & Fragment of $\EIAL$ & Notation \\\hline\hline
non-expanding & ${!}$-for\-mu\-lae should be of the form  & $\EIALne$\\
& ${!}((b_1 \cdot \ldots \cdot b_n) \BS (c_1 \cdot \ldots \cdot c_m))$ where $0 \leqslant m \leqslant n$ &\\\hline
monoidal & ${!}$-for\-mu\-lae should be of the form & $\EIALm$ \\
inequations & ${!}((b_1 \cdot \ldots \cdot b_n) \BS (c_1 \cdot \ldots \cdot c_m))$ & \\\hline
$*$-free & ${!}$-formulae should not contain ${}^*$ & $\EIALminus$\\\hline
\end{tabular}
\end{center}
Complexitywise, the problems in the right column are equal to or harder than the ones in the left one; complexity also grows from top to bottom.
In what follows, we shall prove lower bounds for reasoning from hypotheses and upper bounds for the corresponding fragments of $\EIAL$. Due to the embeddings provided by Theorem~\ref{Th:deduction}, this will give exact complexity bounds for both.

%%

%\newpage
\subsection{Auxiliary dyadic calculus}\label{S:dyadic}

In our lower bound argument for $\EIALm$ (and, stronger, for reasoning from monoidal inequations, Section~\ref{S:Tikhon}) and in our upper bound argument for $\EIALne$ (Section~\ref{S:Stepan}) we shall need some reorganisation of proof trees in order to simplify their analysis. Namely, we want to get rid of superfluous usage of rules operating ${!}$ (most importantly, the contraction rule) and we want each formula of the form $(b_1 \cdot \ldots \cdot b_n) \BS (c_1 \cdot \ldots \cdot c_m)$ to be decomposed immediately above the $({!}L)$ rule which puts it under ${!}$. (Recall that these are the only formulae allowed to be put under ${!}$ in the fragment $\EIALm$.)

These goals could have been achieved in an {\em ad hoc}
way, by proving a series of lemmata using transfinite induction. We prefer, however, to perform the needed transformations more systematically, by presenting an alternative sequent calculus for $\EIALm$ (and thus also for $\EIALne$) and proving that it is equivalent to the original system. 
The idea of this reformulation comes from linear logic, namely, from Andreoli's work on {\em focusing}~\cite{Andreoli}. The idea of focusing is to apply rules eagerly to a given formula, while it is possible, and to control the usage of structural rules.

Following Andreoli, we reformulate $\EIALm$ as a {\em dyadic calculus,} separating ${!}$-for\-mu\-lae in a designated zone of the antecedent. In order to keep things simple, we refrain from constructing a proper focused system for the whole $\EIALminus$. The only place where we need focusing is decomposition of ${!}$-for\-mu\-lae, and since in $\EIALm$ their form is very restricted, we easily formulate a focused rule for them (the rule $(A)_\mathrm{d}$ below).

We start with a modified definition of sequent, suitable for our dyadic calculus.

\begin{definition}
A \emph{dyadic sequent}, or {\em d-sequent,} is an expression of the form ${!}\Xi; \Gamma \yields C$, where $C$ is a formula, $\Gamma$ is a sequence of formulae, and ${!}\Xi$ is a finite set of ${!}$-formulae, each of the form ${!}((b_1 \cdot \ldots \cdot b_n) \BS (c_1 \cdot \ldots \cdot c_m))$.
\end{definition}

Since ${!}\Xi$ is a set, not a multiset or sequence, we shall have intuitionistic-style behaviour for ${!}$-for\-mu\-lae.

Now we construct the dyadic version of $\EIALm$, denoted by $\dEIALm$, in the following way. For all axioms and rules, except those operating with ${!}$, we just add ``${!}\Xi$;'' to each antecedent. For example, the $(\mathrm{Id})$ axiom becomes ${!}\Xi; A \yields A$, the dyadic version of $({}^* R_n)$ is

\[
\infer[({}^* R_n)_{\mathrm{d}}]
{{!}\Xi; \Pi_1, \ldots, \Pi_n \yields A^*}
{{!}\Xi; \Pi_1 \yields A & \ldots & {!}\Xi; \Pi_n \yields A}
\]
and so on. The dyadic version of the $\omega$-rule $({}^* L_\omega)_\mathrm{d}$ is constructed in the same way. Notice that the dyadic version of axioms incorporates weakening, and for multiplicative rules like $({}^* R_n)_{\mathrm{d}}$ and $({\cdot} R)_{\mathrm{d}}$ we also have a hidden application of contraction.

The dyadic rules for ${!}$ are as follows:
\[
\infer[({!}L)_{\mathrm{d}}]
{{!}\Xi; \Gamma, {!}A, \Delta \yields C}
{{!}\Xi \cup \{ {!} A \} ; \Gamma, \Delta \yields C}
\qquad
\infer[({!}R)_{\mathrm{d}}]
{{!}\Xi; \yields {!}B}
{{!}\Xi; \yields B}
\]
\[
\infer[(A)_{\mathrm{d}}]
{\{{!}((b_1 \cdot \ldots \cdot b_n) \BS (c_1 \cdot \ldots \cdot c_m)) \} \cup {!}\Xi; \Gamma, b_1, \ldots, b_n, \Delta \yields C}
{\{{!}((b_1 \cdot \ldots \cdot b_n) \BS (c_1 \cdot \ldots \cdot c_m)) \} \cup {!}\Xi; \Gamma, c_1, \ldots, c_m, \Delta \yields C} 
\]

The rules $({!}L)_{\mathrm{d}}$ and $({!}R)_{\mathrm{d}}$ are the new left and right introduction rules for~${!}$. The $(A)_{\mathrm{d}}$ rule
allows using the formula $B = (b_1 \cdot \ldots \cdot b_n) \BS (c_1 \cdot \ldots \cdot c_m)$, i.e., replacing $c_1, \ldots, c_m$ with $b_1, \ldots, b_n$, provided that ${!}B$ belongs to the ${!}$-zone of the d-se\-qu\-ent. This is actually a combination of structural rules for ${!}B$ (contraction to copy and permutation to move) and the rules decomposing division and product.  The letter ``$A$'' comes from Andreoli's term ``absorption,'' which means taking a formula 
$B = (b_1 \cdot \ldots \cdot b_n) \BS (c_1 \cdot \ldots \cdot c_m)$ and absorbing it into the $!$-zone. Our $(A)_{\mathrm{d}}$ rule, however, also includes decomposition of division and product, so we rather read ``$A$'' as ``application'' (of a formula ${!}B$ from the $!$-zone).

Finally, $\dEIALm$ includes the following two  cut rules:
\[
\infer[(\mathrm{Cut})_{\mathrm{d}1}]
{{!}\Xi; \Gamma, \Pi, \Delta \yields C}
{{!}\Xi; \Pi \yields A & {!}\Xi; \Gamma, A, \Delta \yields C}
\qquad
\infer[(\mathrm{Cut})_{\mathrm{d}2}]
{{!}\Xi; \Gamma \yields C}
{{!}\Xi ; \yields B &  {!}\Xi \cup \{ {!}B \}; \Gamma \yields C}
\]
(The second rule is meaningful when ${!}B \notin {!}\Xi$, otherwise it trivialises.)

The main properties of $\dEIALm$ are cut elimination and equivalence to the original system $\EIALm$. We shall prove them below in Section~\ref{sec-dyadic}, along with some other properties of $\dEIALm$.

\subsection{Inductive definitions}\label{S:inductive}

For each infinitary calculus $\mathbf{L}$, the collection of all sequents derivable in $\mathbf{L}$ can be viewed as the least fixed point of a suitable monotone operator (on sets of sequents) corresponding to $\mathbf{L}$.\footnote{Here
the notion of sequent depends on the choice of $\mathbf{L}$. For example, in case $\mathbf{L} = \dEIALm$ we restrict ourselves to dyadic sequents, as defined in Section~\ref{S:dyadic}.}
For expository purposes, we are going to briefly review this approach.

To simplify our discussion, let
\begin{align*}
\Ord\ &:=\ \textrm{the class of all ordinals} ,\\
\LOrd\ &:=\ \textrm{the class of all limit ordinals} ,\\
\COrd\ &:=\ \textrm{the class of all constructive ordinals} .
\end{align*}
The least element of $\Ord \setminus \COrd$ is traditionally called the \emph{Church--Kleene ordinal}, and denoted by $\omega_1^\mathrm{CK}$. See \cite[\S\S11.7--8]{Rogers-1967} for more on constructive ordinals.

Let $F$ be a monotone function from $\mathcal{P} \left( \omega \right)$ to $\mathcal{P} \left( \omega \right)$, i.e.\ for any $S, T \subseteq \omega$,
\[
S \subseteq T \quad \Longrightarrow \quad
{F \left( S \right)} \subseteq {F \left( T \right)} .
\]
Then for each $S \subseteq \omega$ we can inductively define
\[
{F^{\alpha} \left( S \right)}\ :=\
\begin{cases}
S                                                       &\text{if} ~\, \alpha = 0 ,\\
{F \left( F^{\beta} \left( S \right) \right)}           &\text{if} ~\, \alpha = \beta +1 ,\\
{\bigcup_{\beta < \alpha} {F^{\beta} \left( S \right)}} &\text{if} ~\, \alpha \in \LOrd \setminus \left\{ 0 \right\} .
\end{cases}
\]
Evidently, the resulting transfinite sequence is monotone as a class function from $\Ord$ to $\mathcal{P} \left( \omega \right)$, i.e.\ $\alpha \leqslant \beta$ implies $F^{\alpha} \left( S \right) \subseteq F^{\beta} \left( S \right)$. This observation quickly leads to:

\begin{folk} \label{folk-id-1}
Let $F: \mathcal{P} \left( \omega \right) \rightarrow \mathcal{P} \left( \omega \right)$ be monotone. Then for every $S \subseteq \omega$, if $S \subseteq F \left( S \right)$, there exists $\alpha \in \mathsf{Ord}$ such that $F^{\alpha+1} \left( S \right) = F^{\alpha} \left( S \right)$~--- so $F^{\alpha} \left( S \right)$ is the least fixed point of $F$ containing $S$.
\end{folk}

Let $F$ be as above. For brevity, denote
\[
\mathtt{LFP} \left( F \right)\ :=\
{\text{the least fixed point of}\ F\ (\text{containing}\ \varnothing)} ,
\]
which exists by Folklore~\ref{folk-id-1}. By the \emph{closure ordinal} of $F$ we mean the least $\alpha \in \mathsf{Ord}$ such that $F^{\alpha+1} \left( \varnothing \right) = F^{\alpha} \left( \varnothing \right)$. So this ordinal indicates how many steps are needed to reach $\mathtt{LFP} \left( F \right)$. In addition, for each $n \in \mathtt{LFP} \left( F \right)$, the least $\alpha \in \Ord$ such that $n \in F^{\alpha + 1} \left( \varnothing \right)$ is called the \emph{rank} of $n$ with respect to $F$. Obviously, the closure ordinal of $F$ equals the supremum of the ranks of the members of $\mathtt{LFP} \left( F \right)$ (with respect to $F$).

Denote by $\mathfrak{N}$ the standard model of arithmetic. It is convenient to assume that the sig\-n\-a\-tu\-re of $\mathfrak{N}$ contains a symbol for every computable function or relation. We shall be concerned with monotone operators definable in $\mathfrak{N}$. Take $\mathcal{L}_2$ to be the language of monadic second-order logic based on the signature of $\mathfrak{N}$.\footnote{Of
course, the restriction to monadic formulae is~not essential, for there are computable bijections from $\omega \times \omega$ onto $\omega$.}
Thus $\mathcal{L}_2$ includes two sorts of variables:
\begin{itemize}

\item \emph{individual variables} $x$, $y$, \dots, intended to range over $\omega$;

\item \emph{set variables} $X$, $Y$, \dots, intended to range over the powerset of $\omega$.

\end{itemize}
Accordingly, one needs to distinguish between \emph{individual} and \emph{set quantifiers}, viz.\
\[
{\exists x},\ {\forall x},\ {\exists y},\ {\forall y},\ \dots
\quad \text{and} \quad
{\exists X},\ {\forall X},\ {\exists Y},\ {\forall Y},\ \dots
\]
Let $n \in \omega \setminus \left\{ 0 \right\}$. Recall that an $\mathcal{L}_2$-for\-mu\-la is in $\Sigma^0_n$ ($\Pi^0_n$) if{f} it has the form
\[
\underbrace{{\exists \vec{x}_1}\, {\forall \vec{x}_2}\, \dots\, \vec{x}_n}_{n-1\ \text{alternations}}\, \Psi
\quad
( \text{respectively}\
\underbrace{{\forall \vec{x}_1}\, {\exists \vec{x}_2}\, \dots\, \vec{x}_n}_{n-1\ \text{alternations}}\, \Psi
)
\]
where $\vec{x}_1$, $\vec{x}_2$, \ldots, $\vec{x}_n$ are tuples of individual variables, and $\Psi$ is quantifier-free. A subset of $\omega$ belongs to $\Sigma^0_n$ ($\Pi^0_n$) if{f} it is definable in $\mathfrak{N}$ by a $\Sigma^0_n$-for\-mula ($\Pi^0_n$-for\-mula).

Given an $\mathcal{L}_2$-for\-mu\-la $\Phi \left( x, X \right)$ whose only free variables are $x$ and $X$, we denote by $\left[ \Phi \right]$ the fol\-low\-ing function from $\mathcal{P} \left( \omega \right)$ to $\mathcal{P} \left( \omega \right)$:
\[
{\left[ \Phi \right] \left( S \right)}\ :=\ {\left\{ n \in \omega \mid \mathfrak{N} \models \Phi \left( n, S \right) \right\}} .
\]
Now call $\Phi \left( x, X \right)$ \emph{positive} if{f} no free occurrence of $X$ in $\Phi$ is in the scope of an odd number of nested negations, provided we treat $\rightarrow$ as defined using $\neg$ and $\vee$. Evidently,
\[
{\Phi \left( x, X \right)} \enskip \text{is positive}
\quad \Longrightarrow \quad
{\left[ \Phi \right]} \enskip \text{is monotone} .
\]
We say that a function $F$ from $\mathcal{P} \left( \omega \right)$ to $\mathcal{P} \left( \omega \right)$ is \emph{elementary} if{f} $F = \left[ \Phi \right]$ for some positive~$\mathcal{L}_2$-for\-mu\-la $\Phi \left( x, X \right)$ with no~set quantifiers. See \cite{Moschovakis-1974} for more on operators of this kind.\footnote{In
particular, it is known that whenever $F$ is elementary, then its least fixed point belongs to $\Pi^1_1$, and its closure ordinal is less than or equal to $\omega_1^{\mathrm{CK}}$. But these bounds need~not, in general, be precise.}

Consider an infinitary calculus $\mathbf{L}$, such as $\EIALminus$. We associate with $\mathbf{L}$ its \emph{immediate de\-ri\-va\-bi\-lity operator} $\mathscr{D}_{\mathbf{L}}$ on sets of sequents as follows: for any sequent $s$ and set of sequents $S$ (in the language of $\mathbf{L}$),
\[
s\ \in\ {\mathscr{D}_{\mathbf{L}} \left( S \right)}
\quad \Longleftrightarrow \quad
\begin{array}{c}
s\ \text{is an element of}\ S\ \text{or}\ s\ \text{can be obtained from}\\ %[0.2em]
\text{elements of}\ S\ \text{by one application of some rule of}\ \mathbf{L} .
\end{array}
\]
Notice that the least fixed point of $\mathscr{D}_{\mathbf{L}}$ coincides with the collection of all sequents derivable in $\mathbf{L}$. Assuming some effective G\"{o}del numbering for sequents, we can identify $\mathscr{D}_{\mathbf{L}}$ with a function from $\mathcal{P} \left( \omega \right)$ to $\mathcal{P} \left( \omega \right)$.  For convenience, the immediate derivability operator of $\EIALminus$ will be denoted by $\mathscr{D}^{-}$. It is straightforward to check that $\mathscr{D}^{-}$ is elementary. If $s$ is a sequent derivable in $\EIALminus$, we shall write $\mathrm{rank}^{-} \left( s \right)$ for the rank of $s$ with respect to $\mathscr{D}^{-}$.

%%%

%%%%

%\newpage
\subsection{The hyperarithmetical hierarchy up to \texorpdfstring{$\omega^\omega$}{omega to omega}} \label{subsec-ha}

Given $S \subseteq \omega$, take $\mathtt{U}^S$ to be one's favourite partial $S$-com\-pu\-table two-place function on $\omega$ that is universal for the class of all partial $S$-com\-pu\-table one-place functions on $\omega$. For each $k \in \omega$ we use $\mathtt{U}^S_k$ to denote the $k$-th projection of $\mathtt{U}^S$, i.e.\
\[
\mathtt{U}^S_k\ :=\
{\lambda m}. {\left[ \mathtt{U}^S \left( k, m \right) \right]} .
\]
If $S = \varnothing$, the superscript $S$ in $\mathtt{U}^S$ may be omitted. Further, we write $\leqslant$ for many-one reducibility and $\equiv$ for many-one equivalence. So for any $S, T \subseteq \omega$:
\begin{align*}
S \leqslant T
\quad &\Longleftrightarrow \quad
\text{there exists}\ {k \in \omega}\ \text{such that}\ \mathtt{U}_k\ \text{is total}\ \text{and}\ {S = {\left( \mathtt{U}_k \right)}^{-1} \left[ T \right]} ;\\
S \equiv T
\quad &\Longleftrightarrow \quad
{S \leqslant T}\ \text{and}\ {T \leqslant S} .
\end{align*}
In effect, `many-one' may be replaced `one-one', but in the present context, this is primarily the matter of taste; see \cite[Chapter~7]{Rogers-1967} for more information.

Take $\mathtt{J}$ to be the Turing jump operator on the powerset of $\omega$, which can be defined by
\[
{\mathtt{J} \left( S \right)}\ :=\
{\left\{
k \in \omega \mid
{\mathtt{U}^S \left( k, k \right)} ~ \text{converges}
\right\}} .
\]
Recall the following characterisation of the arithmetical hierarchy of subsets of~$\omega$.

\begin{folk} \label{folk-arithm}
Let $n \in \omega \setminus \left\{ 0 \right\}$. Then for every $S \subseteq \omega$,
\[
S ~ \text{belongs to} ~ \Sigma^0_n
\quad \Longleftrightarrow \quad
{S ~ \text{is many-one reducible to} ~ \mathtt{J}^n \left( \varnothing \right)}
\]
where $\mathtt{J}^n \left( \varnothing \right)$ denotes the result of applying the jump operator $n$ times to $\varnothing$. Moreover, the sets belonging to $\Pi^0_n$ are the complements of those belonging to~$\Sigma^0_n$.
\end{folk}

Roughly speaking, the \emph{hyperarithmetical hierarchy} is obtained by iterating $\mathtt{J}$ over $\COrd$. To do this properly, a special system of notation for $\COrd$, called \emph{Kleene's $\mathcal{O}$}, is applied. However, we shall only need the first $\omega^\omega$ iterations of~$\mathtt{J}$. Hence it is reasonable to use a simpler system of notation, which will be described shortly.

Let $\Ns$ be the collection of all sequences of natural numbers with limit $0$, i.e.\
\[
\Ns\ :=\
{\left\{
\left( m_0, m_1, \ldots \right)  \mid
{\left( \exists k \right)}\, {\left( \forall i \geqslant k \right)}\, {\left( m_i = 0 \right)}
\right\}} .
\]
Consider the binary relation $\prec$ on $\Ns$ defined by
\[
{\left( m_0, m_1, \ldots \right) \prec \left( n_0, n_1, \ldots \right)}\
:\Longleftrightarrow\
{\exists k}\, {\left( {m_k < n_k} \wedge {\left( \forall i > k \right)}\, {\left( m_i = n_i \right)} \right)} .
\]
Evidently, $\prec$ is a well-ordering of $\Ns$. Furthermore, $\Ns$ and $\omega^{\omega}$ are isomorphic via the function $\nu$ from $\Ns$ to $\omega^{\omega}$ given by
\[
{\nu \left( \left( m_0, m_1, m_2, \ldots \right) \right)}\ :=\
{\ldots + \omega^2 \cdot m_2 + \omega \cdot m_1 + m_0} .
\]
Of course, we can identify $\Ns$ with a computable subset of $\omega$, say, by using the function $\rho$ from $\Ns$ to $\omega$ given by
\[
{\rho \left( \left( m_0, m_1, m_2, \ldots \right) \right)}\ :=\
{p_0}^{m_0} \cdot {p_1}^{m_1} \cdot {p_2}^{m_2} \cdot \ldots
\]
where $p_k$ is the $\left( k + 1 \right)$-st prime number. Then $\rho^{-1} \circ \nu$ turns out to be a univalent computable system of notation for $\omega^{\omega}$, which is reducible to Kleene's $\mathcal{O}$ in a suitable way; see \cite[\S\,11.7]{Rogers-1967} for details. If $\alpha < \omega^{\omega}$, we write $\sharp \alpha$ for $\rho \left( \nu^{-1} \left( \alpha \right) \right)$.

Now at the heart of our smaller hierarchy is an indexed family $\langle {H \left( \alpha \right)}: \alpha \leqslant \omega^{\omega} \rangle$ of subsets of $\omega$ such that
\[
{H \left( \alpha \right)}\ :=\
\begin{cases}
\varnothing &\text{if}\ \alpha = 0\\
{\mathtt{J} \left( H \left( \beta \right) \right)} &\text{if}\ \alpha = \beta + 1\\
{\left\{
\mathtt{c}\, ( m, {\sharp \beta} ) \mid
{m \in H \left( \beta \right)}\ \text{and}\ {\beta < \alpha} \right\}} &\text{if}\ {\alpha = \LOrd \setminus \left\{ 0 \right\}}
\end{cases}
\]
where $\mathtt{c}$ denotes one's favourite computable bijection from $\omega \times \omega$ onto $\omega$ (say, the well-known Cantor pairing function); cf.\ \cite[\S\,16.8]{Rogers-1967}. It should be remarked that in Kleene's $\mathcal{O}$ each~const\-ruc\-ti\-ve~ordinal may receive many notations, and if $m$ and $n$ are notations for a constructive limit ordinal $\alpha$, then $H \left( m \right)$ and $H \left( n \right)$ (as defined in \cite[\S\,16.8]{Rogers-1967}) are only Turing equivalent, but not necessarily many-one equivalent. This problem does~not arise in the case of $\rho^{-1} \circ \nu$.

Take $\Sigma^0_0$ --- as well as $\Pi^0_0$ --- to be the collection of all computable subsets of $\omega$. With each non-zero $\alpha \leqslant \omega^{\omega}$ we associate
\[
\Sigma^0_{\alpha}\ :=\
{\left\{
S \subseteq \omega \mid
S \leqslant H \left( \alpha \right)
\right\}}
\quad \text{and} \quad
\Pi^0_{\alpha}\ :=\
{\left\{
\omega \setminus S \mid
S \in \Sigma^0_\alpha
\right\}} . %or {\left\{ S \subseteq \omega \mid S \leqslant \omega \setminus H \left( \alpha \right) \right\}
\]
This gives us (a variant of) the \emph{hyperarithmetical hierarchy up to $\omega^{\omega}$}, whose initial segment of type $\omega$ coincides with the arithmetical hierarchy by Folklore~\ref{folk-arithm}. Further, since both $H \left( \omega \right)$ and $\omega \setminus H \left( \omega \right)$ are many-one equivalent to complete first-order arithmetic, $\Sigma^0_{\omega}$ and $\Pi^0_{\omega}$ coincide. To get a larger picture, we have the following.

\begin{folk} \label{folk-hierarchy}
For any $\alpha, \beta \leqslant \omega^{\omega}$:
\begin{itemize}

\item $\Sigma^0_\alpha \cup \Pi^0_\beta \subsetneq \Sigma^0_{\alpha +1} \cap \Pi^0_{\beta +1}$;

\item if $\alpha \not \in \LOrd$, then
$\Sigma^0_{\alpha} \setminus \Pi^0_{\alpha} \ne
\varnothing$ and $\Pi^0_{\alpha} \setminus \Sigma^0_{\alpha} \ne
\varnothing$;

\item if $\alpha \in \LOrd$, then
$\Sigma^0_{\alpha} = \Pi^0_{\alpha}$.

\end{itemize}
\end{folk}

Finally, given $\alpha \leqslant \omega^{\omega}$, we say that $S \subseteq \omega$ is
\emph{$\Sigma^0_{\alpha}$-com\-p\-le\-te} if $S \equiv H \left( \alpha \right)$, and \emph{$\Pi^0_{\alpha}$-com\-p\-le\-te} if $S \equiv \omega \setminus H \left( \alpha \right)$. Clearly, if $\alpha$ is limit, then the two notions coincide.

%%%

%\newpage
\section{Properties of the dyadic system}\label{sec-dyadic}

In this section, we shall prove certain properties of the dyadic system~$\dEIALm$. This system was introduced in Section~\ref{S:dyadic} above and will be used as a technical device in the proofs in Sections~\ref{S:Tikhon} and~\ref{S:Stepan} below. Namely, the usage of this specific dyadic system enables some amount of Andreoli-style~\cite{Andreoli} focusing, which makes proof analysis much easier. The most important properties of $\dEIALm$ are cut elimination and equivalence to the original system $\EIALm$.

\begin{prop}\label{dyadic-cut-elim}
Any sequent derivable in $\dEIALm$ can be derived without using $(\mathrm{Cut})_{\mathrm{d}1}$ and $(\mathrm{Cut})_{\mathrm{d}2}$.
\end{prop}

\begin{proof}
The proof is by the same argument as Palka's proof of cut elimination for $\IAL$~\cite[Theorem~3.1]{Palka2007}. (Thus, it is simpler than the argument from~\cite[Theorem~4.2]{KuznetsovSperanski2022} for cut elimination in $\EIAL$, where a form of Mix rule was used; here it is unnecessary.) 

In order to eliminate one cut (either $(\mathrm{Cut})_{\mathrm{d}1}$ or $(\mathrm{Cut})_{\mathrm{d}2}$), we suppose that its premises already have cut-free derivations and prove that then the conclusion also has one. (Once we know how to eliminate one cut, a standard transfinite induction argument shows that any derivable sequent is cut-free derivable.)
We proceed by triple induction on the following parameters: (1) the complexity $|A|$ of the formula being cut (complexity is computed in the usual way, as the total number of variable, constant, and connective occurrences); (2) the rank of the right premise of cut; (3) the rank of the left premise of cut. In (2) and (3), ranks are computed w.r.t.\ the cut-free immediate derivability operator.
As usual, we consider possible cases on the lowermost rules in the derivation of premises of cut.

The new interesting cases are those involving the rules $({!}L)_{\mathrm{d}}$, $({!}R)_{\mathrm{d}}$, and $(A)_{\mathrm{d}}$. Other rules are considered exactly as in Palka's proof~\cite[Theorem~3.1]{Palka2007}. Moreover, the cases where one of these lowermost rules is \emph{non-principal,} i.e., does not operate the formula $A$ or $B$ being cut, are also trivial: cut can be propagated upwards, reducing induction parameter (2) or (3). Notice that such propagation is possible for both $(\mathrm{Cut})_{\mathrm{d}1}$ and $(\mathrm{Cut})_{\mathrm{d}2}$. (The latter was not considered by Palka, but the proof transformations are exactly the same.)

Thus, the interesting cases are those where both lowermost rules are principal and are not covered in Palka's proof. That is, we have the following two cases.

\emph{Case 1:} the formula being cut is ${!}A$, introduced by $({!}R)_{\mathrm{d}}$ on the left and $({!}L)_{\mathrm{d}}$ on the right. The cut rule in this case is $(\mathrm{Cut})_{\mathrm{d}1}$.
\[
\infer[(\mathrm{Cut})_{\mathrm{d}1}]
{{!}\Xi; \Gamma, \Delta \yields C}
{\infer[({!}R)_{\mathrm{d}}]{{!}\Xi; \yields {!}A}
{{!}\Xi; \yields A} & 
\infer[({!}L)_{\mathrm{d}}]{{!}\Xi; \Gamma, {!}A, \Delta \yields C}
{{!}\Xi \cup \{ {!} A \}; \Gamma, \Delta \yields C}}
\]
This is transformed into an application of $(\mathrm{Cut})_{\mathrm{d}2}$ with a smaller parameter (1): $A$ has less complexity than ${!}A$.
\[
\infer[(\mathrm{Cut})_{\mathrm{d}2}]
{{!}\Xi; \Gamma, \Delta \yields C}
{{!}\Xi; \yields A & {!}\Xi \cup \{ {!} A \}; \Gamma, \Delta \yields C}
\]

\emph{Case 2:} the left premise is derived by $({\BS}R)_{\mathrm{d}}$ introducing the formula $B = (b_1 \cdot \ldots \cdot b_n) \BS (c_1 \cdot \ldots \cdot c_m)$, and the lowermost rule in the derivation of the right premise is $(A)_{\mathrm{d}}$ which uses ${!}B$ in the ${!}$-zone. The cut rule here is $(\mathrm{Cut})_{\mathrm{d}2}$, and this is the only principal case for this cut rule. (In all other cases it just gets propagated, as noticed above. For the axiom, the contents of the ${!}$-zone does not matter, so $(\mathrm{Cut})_{\mathrm{d}2}$ which removes something from it can be just ignored.) 

The derivation in Case~2 ends as follows (the interesting case is ${!}B \notin {!}\Xi$):
\[
\infer[(\mathrm{Cut})_{\mathrm{d}2}]
{{!}\Xi; \Gamma, b_1, \ldots, b_n, \Delta \yields C}
{\infer[({\BS}R)_{\mathrm{d}}]{{!}\Xi; \yields B}
{{!}\Xi; b_1 \cdot \ldots \cdot b_n \yields c_1 \cdot \ldots \cdot c_m}
&
\infer[(A)_{\mathrm{d}}]
{{!}\Xi \cup \{ {!}B \}; \Gamma, b_1, \ldots, b_n, \Delta  \yields C}
{{!}\Xi \cup \{ {!}B \}; \Gamma, c_1, \ldots, c_m, \Delta \yields C}
}
\]
This gets reconstructed in the following way.
Consider the following two derivations (recall that a double line means several applications of a rule):
\[
\infer[(\mathrm{Cut})_{\mathrm{d}1}]
{{!}\Xi; b_1, \ldots, b_n \yields c_1 \cdot \ldots \cdot c_m}
{\infer=[({\cdot}R)_{\mathrm{d}}]{{!}\Xi; b_1, \ldots, b_n \yields b_1 \cdot \ldots \cdot b_n}{{!}\Xi; b_1 \yields b_1 & \ldots & {!}\Xi; b_n \yields b_n} & 
{!}\Xi; b_1 \cdot \ldots \cdot b_n \yields c_1 \cdot \ldots \cdot c_m} 
\]

\[
\infer=[({\cdot}L)_{\mathrm{d}}]{{!}\Xi; \Gamma, c_1 \cdot \ldots \cdot c_m, \Delta \yields C}
{\infer[(\mathrm{Cut})_{\mathrm{d}2}]{{!}\Xi; \Gamma, c_1, \ldots, c_m, \Delta \yields C}{{!}\Xi; \yields B  & {!}\Xi \cup \{ {!}B \}; \Gamma, c_1, \ldots, c_m, \Delta \yields C}}
\]
(If $m = 0$ or $n = 0$, we use rules for $\U$ instead of the ones for product.) 

In the first derivation, $(\mathrm{Cut})_{\mathrm{d}1}$ is applied to $b_1 \cdot \ldots \cdot b_n$, which has strictly smaller complexity than $B$. Thus, parameter~(1) got reduced, and by induction hypothesis ${!}\Xi; b_1, \ldots, b_n \yields c_1 \cdot \ldots \cdot c_m$ is cut-free derivable.  

In the second derivation, $(\mathrm{Cut})_{\mathrm{d}2}$ is applied with the same parameter~(1) and a smaller parameter~(2): the rank of the right premise got decreased by 1. By induction hypothesis, ${!}\Xi; \Gamma, c_1, \ldots, c_m, \Delta \yields C$ is cut-free derivable, and therefore so is ${!}\Xi; \Gamma, c_1 \cdot \ldots \cdot c_m, \Delta \yields C$.

Finally, we combine the two d-sequents using $(\mathrm{Cut})_{\mathrm{d}1}$:
\[
\infer[(\mathrm{Cut})_{\mathrm{d}1}]
{{!}\Xi; \Gamma, b_1, \ldots, b_n, \Delta \yields C}
{{!}\Xi; b_1, \ldots, b_n \yields c_1 \cdot \ldots \cdot c_m & {!}\Xi; \Gamma, c_1 \cdot \ldots \cdot c_m, \Delta \yields C}
\]
This instance of cut has a smaller parameter~(1), and by induction hypothesis the goal d-sequent is cut-free derivable.
\end{proof}

Before proving that $\dEIALm$ is equivalent to the original system $\EIALm$, let us establish some properties of $\dEIALm$.

\begin{lemm}
If a d-sequent ${!}\Xi; \Gamma \yields C$ is derivable in $\dEIALm$, then so is ${!}\Xi \cup {!}\Xi'; \Gamma \yields C$ for any ${!}\Xi'$.
\end{lemm}

\begin{proof}
Let us take the derivation of ${!}\Xi; \Gamma \yields C$ and add ${!}\Xi'$ to the $!$-zone of each sequent. This results in a valid derivation of ${!}\Xi \cup {!}\Xi'; \Gamma \yields C$.
\end{proof}

This lemma actually says that $\dEIALm$ admits the following weakening rule for the ${!}$-zone, and we may add it as an inference rule without changing the set of derivable (without hypotheses) d-sequents:
\[
\infer[(W)_{\mathrm{d}}]
{{!}\Xi \cup {!}\Xi'; \Gamma \yields C}
{{!}\Xi; \Gamma \yields C}
\]

\begin{lemm}\label{bd-inv}
If a d-sequent ${!}\Xi; \Gamma, {!}A, \Delta \yields C$ is derivable in $\dEIALm$, then so is ${!}\Xi \cup \{ {!}A \}; \Gamma, \Delta \yields C$.
(Notice that $A$, being a formula under ${!}$, is required to be of the form $(b_1 \cdot \ldots \cdot b_n) \BS (c_1 \cdot \ldots \cdot c_m)$.)
\end{lemm}

\begin{proof}
First, we show that the d-sequent $\{{!}A\}; {}\yields A$ is derivable for any formula $A$ of the form $(b_1 \cdot \ldots \cdot b_n) \BS (c_1 \cdot \ldots \cdot c_m)$:
\[
\infer[({\BS}R)_{\mathrm{d}}]
{\{ {!}((b_1 \cdot \ldots \cdot b_n) \BS (c_1 \cdot \ldots \cdot c_m)) \}; \yields (b_1 \cdot \ldots \cdot b_n) \BS (c_1 \cdot \ldots \cdot c_m)}
{ \infer=[({\cdot}L)_{\mathrm{d}}]{\{ {!}((b_1 \cdot \ldots \cdot b_n) \BS (c_1 \cdot \ldots \cdot c_m)) \}; b_1 \cdot \ldots \cdot b_n \yields c_1 \cdot \ldots \cdot c_m}{\infer[(A)_{\mathrm{d}}]{\{ {!}((b_1 \cdot \ldots \cdot b_n) \BS (c_1 \cdot \ldots \cdot c_m)) \}; b_1,  \ldots, b_n \yields c_1 \cdot \ldots \cdot c_m}{\infer=[(\cdot R)_{\mathrm{d}}]{\{ {!}((b_1 \cdot \ldots \cdot b_n) \BS (c_1 \cdot \ldots \cdot c_m)) \}; c_1, \ldots, c_m \yields c_1 \cdot \ldots \cdot c_m}{\{ {!}((b_1 \cdot \ldots \cdot b_n) \BS (c_1 \cdot \ldots \cdot c_m)) \}; c_1 \yields c_1 & \ldots & \ldots ; c_m \yields c_m}}}}
\]
(For $n=1$ or $m=1$, the corresponding product rules get replaced with rules for 1.)

Now we derive ${!}\Xi \cup \{ {!}A \}; \Gamma, \Delta \yields C$ from ${!}\Xi; \Gamma, {!}A, \Delta \yields C$:
\[
\infer[(\mathrm{Cut})_{\mathrm{d}1}]
{{!}\Xi \cup \{ {!}A \}; \Gamma, \Delta \yields C}
{
\infer[({!}R)_{\mathrm{d}}]
{{!}\Xi \cup \{ {!}A \}; {} \yields {!}A}
{\infer[(W)_{\mathrm{d}}]{{!}\Xi \cup \{ {!}A \}; {} \yields A}
{\{ {!}A \}; {} \yields A}} & 
\infer[(W)_{\mathrm{d}}]{{!}\Xi \cup \{ {!} A \}; \Gamma, {!}A, \Delta \yields C}
{{!}\Xi; \Gamma, {!}A, \Delta \yields C}
}
\]
The derivation can be made cut-free by Proposition~\ref{dyadic-cut-elim}.
\end{proof}

This lemma establishes the invertibility of $({!}L)_{\mathrm{d}}$ or, in other words, the admissibility of the following rule:
\[
\infer[({!}L)_{\mathrm{d}}^{\mathrm{inv}}]
{{!}\Xi \cup \{ {!}A \}; \Gamma, \Delta \yields C}
{{!}\Xi; \Gamma, {!}A, \Delta \yields C}
\]

\begin{prop}\label{prop:dyadic-equivalence}
A sequent $\Pi \yields C$ is derivable in $\EIALm$ if and only if ${}; \Pi \yields C$ (with an empty $!$-zone) is derivable in $\dEIALm$.
\end{prop}

\begin{proof}
For the ``if'' direction, we prove a stronger statement:
if a d-sequent ${!}\Xi; \Pi \yields C$ is derivable in $\dEIALm$, then the sequent ${!}\Xi, \Pi \yields C$ is derivable in $\EIALm$. In the latter ${!}\Xi$ stands for the sequence of all formulae contained in ${!}\Xi$, in an arbitrary order. The proof is straightforward transfinite induction on the derivation (we may suppose, due to cut elimination, that this derivation is cut-free). The interesting cases are, of course, $({!}L)_{\mathrm{d}}$, $({!}R)_{\mathrm{d}}$, and
$(A)_{\mathrm{d}}$. These are translated into $\EIALm$ as follows. For $({!}L)_{\mathrm{d}}$ we have
\[
\infer[({!}P_2)]
{{!}\Xi, \Gamma, {!}A, \Delta \yields C}
{{!}\Xi_1, {!}A, {!}\Xi_2, \Gamma, \Delta \yields C}
\]
(putting ${!}A$ to some place inside ${!}\Xi$; here and further ${!}\Xi_1, {!}\Xi_2 = {!}\Xi$). The rule $({!}R)_{\mathrm{d}}$ translates exactly into $({!}R)$. 

Finally, for $(A)_{\mathrm{d}}$ with $B = ((b_1 \cdot \ldots \cdot b_n) \BS (c_1 \cdot \ldots \cdot c_m))$ we construct the fol\-lowing derivation (with 1 instead of product, if $n$ or $m$ is zero):
\[
\infer[({!}C)]
{{!}\Xi_1, {!}B, {!}\Xi_2, \Gamma, b_1, \ldots, b_n, \Delta \yields C}
{\infer[({!}P_1)]{{!}\Xi_1, {!}B, {!}B, {!}\Xi_2, \Gamma, b_1, \ldots, b_n, \Delta \yields C}
{\infer[({!}L)]{{!}\Xi_1, {!}B, {!}\Xi_2, \Gamma, b_1, \ldots, b_n, {!}B, \Delta \yields C}
{\infer[({\BS}L)]{{!}\Xi_1, {!}B, {!}\Xi_2, \Gamma, b_1, \ldots, b_n, (b_1 \cdot \ldots \cdot b_n) \BS (c_1 \cdot \ldots \cdot c_m), \Delta \yields C}{b_1, \ldots, b_n \yields b_1 \cdot \ldots \cdot b_n  & \infer=[({\cdot}L)]{{!}\Xi_1, {!}B, {!}\Xi_2, \Gamma, c_1 \cdot \ldots \cdot c_m, \Delta \yields C}{{!}\Xi_1, {!}B, {!}\Xi_2, \Gamma, c_1, \ldots, c_m, \Delta \yields C}}}}}
\]

Other rules are translated straightforwardly, adding (where needed) permutation rules to move ${!}\Xi$, and contraction rules to merge two  or several copies of ${!}\Xi$ into one. For the axioms, if the ${!}$-zone is non-empty, it is added using the weakening rule ${!}W$.

The ``only if'' direction is also proved by transfinite induction on derivation. However, here we shall essentially use cut in the $\dEIALm$ derivation being constructed. Fortunately, we have cut elimination (Proposition~\ref{dyadic-cut-elim}), so in the end we can get rid of cuts.

Again, axioms and all rules not operating ${!}$ are translated ``as is,'' the ${!}$-zones being empty. The interesting rules are $({!}L)$,  $({!}R)$, $({!}P_{1,2})$, $({!}C)$, and $({!}W)$. 
Mo\-del\-ling of these rules is quite straightforward, given the handy admissible rules $(W)_{\mathrm{d}}$ and $({!}L)_{\mathrm{d}}^{\mathrm{inv}}$ (admissibility established above).

For $({!}L)$, we have
\[
\infer[({!}L)_{\mathrm{d}}]
{; \Gamma, {!}A, \Delta \yields C}
{\infer[(\mathrm{Cut})_{\mathrm{d}1}]{\{ {!}A \} ; \Gamma, \Delta \yields C}
{\{{!}A\}; \yields A & \infer[(W)_{\mathrm{d}}]{\{ {!}A \}; \Gamma, A, \Delta \yields C}{; \Gamma, A, \Delta \yields C}}}
\]
Derivability of $\{{!}A\}; \yields A$ was established in the proof of Lemma~\ref{bd-inv}.

For $({!}R)$, we have
\[
\infer=[({!}L)_{\mathrm{d}}]
{; {!}A_1, \ldots, {!}A_k \yields {!}B}
{\infer[({!}R)_{\mathrm{d}}]
{\{ {!}A_1, \ldots, {!}A_k \} ; \yields {!}B}
{\infer=[({!}L)_{\mathrm{d}}^{\mathrm{inv}}]{\{ {!}A_1, \ldots, {!}A_k \} ; \yields B}
{; {!}A_1, \ldots, {!}A_k \yields B}}}
\]

For $({!}P_1)$ and $({!}P_2)$, we have
\[
\infer[({!}L)_{\mathrm{d}}]
{; \Gamma, {!}A, \Pi, \Delta \yields C}
{\infer[({!}L)_{\mathrm{d}}^{\mathrm{inv}}]
{\{{!}A\}; \Gamma, \Pi, \Delta \yields C}
{; \Gamma, \Pi, {!}A, \Delta \yields C}}
\qquad
\infer[({!}L)_{\mathrm{d}}]
{; \Gamma, \Pi, {!}A, \Delta \yields C}
{\infer[({!}L)_{\mathrm{d}}^{\mathrm{inv}}]
{\{{!}A\}; \Gamma, \Pi, \Delta \yields C}
{; \Gamma, {!}A, \Pi, \Delta \yields C}}
\]

Finally, for $({!}W)$ and $({!}C)$ we have the following
\[
\infer[({!}L)_{\mathrm{d}}]
{;\Gamma, {!}A, \Delta \yields C}
{\infer[(W)_{\mathrm{d}}]{\{ {!}A \} ; \Gamma, \Delta \yields C}
{; \Gamma, \Delta \yields C}}
\qquad
\infer[({!}L)_{\mathrm{d}}]
{; \Gamma, {!}A, \Delta \yields C}
{\infer[({!}L)_{\mathrm{d}}^{\mathrm{inv}}]
{\{ {!}A \}; \Gamma, \Delta \yields C}
{\infer[({!}L)_{\mathrm{d}}^{\mathrm{inv}}]{\{ {!}A \}; \Gamma, {!}A, \Delta \yields C}
{; \Gamma, {!}A, {!}A, \Delta \yields C}}}
\]
In the derivation modelling $({!}C)$, the upper instance of $({!}L)_{\mathrm{d}}^{\mathrm{inv}}$ is applied with ${!}\Xi = \varnothing$ and the lower one is applied with ${!}\Xi = \{ {!}A \}$, having $\{ {!}A \} \cup \{ {!}A \} = \{ {!}A \}$.
\end{proof}

\section{A lower bound argument}\label{S:Tikhon}

Let us start with establishing the complexity lower bound for  reasoning from $\ast$-free hypotheses in $\ast$-con\-ti\-nu\-ous action lattices. Recall that, according to Theorem \ref{theo-strong-completeness}, semantic entailment in $\ast$-con\-ti\-nu\-ous action lattices is equivalent to derivability from hypotheses in $\IAL$. The main theorem of this section establishes the $\Sigma^0_{\omega^{\omega}}$-hard\-ness of derivability from $*$-free hypotheses in $\IAL$ in the following strong sense.
\begin{theo}\label{Th:lower-bound}
    There exists a finite set $\Hc$ of monoidal inequations such that $H( \omega^\omega)$ is many-one reducible to the set of sequents derivable in $\IAL$ from $\Hc$.
\end{theo}
Recall that monoidal inequations are represented as sequents of the form $x_1,\ldots,x_n \yields y_1 \mconj \ldots \mconj y_m$ where $x_i,y_j$ are variables. Compare this with non-expanding hypotheses, for which, as we shall prove in Section \ref{S:Stepan}, the derivability problem is $\Pi^0_1$-comp\-le\-te.

The proof of Theorem \ref{Th:lower-bound} uses techniques similar to those used in \cite{Pshenitsyn2024} to show that the derivability problem for infinitary action logic with multiplexing is $\Sigma^0_{\omega^{\omega}}$-hard (although details of proofs are different). We shall construct a reduction from a $\Sigma^0_{\omega^{\omega}}$-comp\-le\-te set $\mathbf{T}$, which consists of indices of computable infinitary formulae with rank less than $\omega^\omega$ in the simplest language (they are called {\em infinitary propositional sentences} in \cite{Montalban2022}), to the set of sequents derivable in $\IAL$ from $\Hc$. The definition of infinitary propositional sentences we provide below is similar to that from \cite[Section 7]{AshKnight2000}, yet adjusted to match other definitions in this paper.

Let $\langle \cdot,\cdot,\cdot\rangle$ be a computable bijection from $\omega \times \omega \times \omega$ onto $\omega$, say, $\langle i,j,k \rangle := \mathtt{c}(i,\mathtt{c}(j,k))$. Also let $\overline{\varepsilon} := 1 - \varepsilon$, for $\varepsilon \in \{0,1\}$.

\begin{definition}\label{Df:infinitary_propositional_sentences}
    For $\alpha < \omega^{\omega}$, let 
    \[
    S^{\Sigma}_{\alpha}\ :=\
    {\left\{\langle 0, \sharp \alpha, e \rangle \mid e \in \omega \right\}}
    \quad \mbox{and} \quad
    S^{\Pi}_{\alpha}\ :=\
    {\left\{\langle 1, \sharp \alpha, e \rangle \mid e \in \omega \right\}} .
    \]
    Elements of these sets are called \emph{$\Sigma_\alpha$-in\-di\-ces} and \emph{$\Pi_\alpha$-in\-di\-ces} respectively. The index $\langle \varepsilon, \sharp 0, e \rangle$ represents the sentence $\bot$ if $e=0$ and $\top$ otherwise (for $\varepsilon=0,1$). For $\alpha > 0$, the index $\langle 0, \sharp \alpha, e \rangle$ corresponds to the infinitary disjunction of infinitary propositional sentences such that the set of their indices is
    \[
    \left\{ k \mid {\UCF\left(e,k\right)}\ \text{converges and}\ {\left( \exists \beta < \alpha \right)}\, {\left( \exists e^\prime \right)}\, {\left( k = \langle 1, \sharp \beta, e^\prime \rangle \right)} \right\} .
    \]
    Symmetrically, the index $\langle 1, \sharp \alpha, e \rangle$ represents the infinitary conjunction of infinitary propositional sentences whose indices form the set
    \[
    \left\{k \mid {\UCF(e,k)}\ \text{converges and}\ {\left( \exists \beta < \alpha \right)}\, {\left( \exists e^\prime \right)}\, {\left( k = \langle 0, \sharp \beta, e^\prime \rangle \right)} \right\} .
    \]
    Finally, let $\mathbf{T}$ be the set of indices of all true infinitary propositional sentences defined thusly.
\end{definition}
\begin{prop}\label{prop:hardness-T}
    The set $H \left( \omega^{\omega} \right)$ is many-one reducible to $\mathbf{T}$.
\end{prop}
This proposition is proved in essentially the same way as \cite[Lemma V.15]{Montalban2022}; cf. also \cite[Theorem 7.9]{AshKnight2000}. To make the paper self-contained, we provide a proof of this proposition in the Appendix.

Another ingredient we need is a representation of computable functions as string rewriting systems in a way convenient for our purposes. Recall that a string rewriting system $\SR$ over an alphabet $A$ is a finite set of rules of the form $w_1 \to w_2$ where $w_1$, $w_2$ are arbitrary strings from $A^\ast$. We write $u w_1 v \Rightarrow_{\SR} u w_2 v$ if $(w_1 \to w_2) \in \SR$. Also, recall that a Turing machine is a tuple $\TM = (Q,C, T, \qbgn, \qacc)$ where $Q$ is a finite set of states; $C$ is the alphabet of tape symbols contaning the blank symbol $\lambda \in C$; $T \subseteq \left(Q \setminus \{\qacc\} \right) \times C \times Q \times C \times \{R,L,N\}$ is the transition relation; $\qbgn \in Q$ is the initial state; $\qacc \in Q$ is the accepting state. 

It is well known that Turing machines can be encoded by string rewriting systems \cite{BookOtto1993}. Indeed, a Turing machine configuration can be represented as a string $uaqv$ where $uav$ is a part of the tape containing all non-blank symbols, $q$ is the state of a Turing machine, and the head points to the distinguished symbol $a$. We assume that, when a Turing machine halts in the accepting state $\qacc$, its head points to the rightmost blank symbol before the output. 

Let us present how Turing machines are transformed into string rewriting systems; this particular transformation is convenient for the further construction. 
\begin{definition}\label{def:SR-from-TM}
    Let $f \subseteq A^\ast \times B^\ast$ be a partial computable function calculated by a Turing machine $\TM = (Q,C, T, \qbgn, \qacc)$ where $C \supseteq A \cup B$; let $l,l^\prime,r,e$ be symbols not from $C \cup Q$. Then, let us define a string rewriting system $\SR_{l,r,e}(\TM)$ that consists of the following rules: 
    \begin{align*}
    & a q \to b q^\prime, &&\mbox{for each } (q,a,q^\prime,b,N) \in T; \\
    & l q \to l b q^\prime, &&\mbox{for each }  (q,\lambda,q^\prime,b,N) \in T; \\
    & a q \to q^\prime b, &&\mbox{for each }  (q,a,q^\prime,b,L) \in T; \\
    & l q \to l q^\prime b, &&\mbox{for each } (q,\lambda,q^\prime,b,L) \in T; \\
    & a q c \to b c q^\prime, &&\mbox{for each }  (q,a,q^\prime,b,R) \in T,\ c \in C; \\
    & a q r \to b \lambda q^\prime r, &&\mbox{for each }  (q,a,q^\prime,b,R) \in T; \\
    & l q r \to l b \lambda q^\prime r, &&\mbox{for each }  (q,\lambda,q^\prime,b,R) \in T;  \\
    & \lambda \qacc \to \qacc; \\
    & l \qacc \to l l^\prime \qacc; \\
    & l^\prime \qacc c \to c l^\prime \qacc, &&\mbox{for each }  c \in B; \\
    & l^\prime \qacc \lambda \to l^\prime \qacc; \\
    & l^\prime \qacc r \to e.
    \end{align*}
\end{definition}
Here $l,r$ are ``borders'' of a tape modelled as a string and $e$ is a distinguished symbol that indicates that the computation has ended. The symbol $l^\prime$ plays a technical role, and we do not include it in the notation $\SR_{l,r,e}(\TM)$. Note that the left-hand side of each rule in $\SR$ contains a symbol from $Q$. The correspondence between $\TM$ and $\SR = \SR_{l,r,e}(\TM)$ can be described as follows.
\begin{prop}\label{prop:Church-Turing}
    For any $u \in A^\ast$ and $v \in D^\ast$ where $D = C \cup Q \cup \{l, l^\prime, r, e\}$, we have $l u \qbgn r \Rightarrow_{\SR}^\ast v e$ if and only if $v = l f(u)$.
\end{prop}
\begin{proof}
    One can examine Definition \ref{def:SR-from-TM} and verify that any rule of $\SR = \SR_{l,r,e}(\TM)$ applied to a string of the form $l w_1 a q w_2 r$ (where $q\ne \qacc$ and $w_1,w_2 \in C^\ast$) corresponds to a computation step of $\TM$. Assume that one reaches an accepting configuration $l \lambda^m \qacc w \lambda^n r$ (where $w \in B^\ast$). Then the only possible derivation ending by a string with the suffix $e$ looks as follows:
    \begin{multline*}
        l \lambda^m \qacc w \lambda^n r
        \ \Rightarrow^\ast_{\SR} \ 
        l \qacc w \lambda^n r
        \ \Rightarrow_{\SR}\ 
        l l^\prime \qacc w \lambda^n r
        \\ \Rightarrow^\ast_{\SR} \ 
        l w l^\prime \qacc \lambda^n r
        \ \Rightarrow^\ast_{\SR}\ 
        l w l^\prime \qacc r
        \ \Rightarrow_{\SR} \ 
        l w e.
    \end{multline*}
    So, indeed, if a string of the form $ve$ is obtained from $lu\qbgn r$, then $v=lw$ where $w$ is the output of $\TM$ on the input $u$.
\end{proof}

\begin{definition}\label{Df:check-functions}
    Let us fix two symbols $\VARsym_1$, $\VARsym_2$ and define functions $\fun_0$ and $\fun_1$, called the {\em base-check function} and the {\em step-check function} respectively.
    \begin{enumerate}[i.]

    \item The domain of $\fun_{0}$ is a subset of $\{\VARsym_{1}\}^\ast$. If $x$ is a $\Sigma_0$- or a $\Pi_0$-index and $x \in \mathbf{T}$, then $\fun_{0}\left(\VARsym_{1}^x\right) := \VARsym_{1}$ (here $p_1^x = p_1 \ldots p_1$ with $x$ copies of $p_1$). If $x$ is a $\Sigma_\alpha$- or a $\Pi_\alpha$-index for $\alpha>0$, then $\fun_{0} (\VARsym_{1}^x) := \VARsym_{1}^x \VARsym_{2}$. Otherwise, $\fun_{0}(\VARsym_{1}^x)$ is undefined.

        \item The domain of $\fun_{1}$ is $\{\VARsym_{1},\VARsym_2\}^\ast$. Assume that an input $w$ is of the form $w = \VARsym_{1}^{x_1} \VARsym_{2}^{x_2}$ such that:
		\begin{itemize}
            \item $x_1 = \langle \varepsilon, \sharp \alpha, e\rangle$ is a $\Sigma_\alpha$- or a $\Pi_\alpha$-index for $\alpha > 0$;
            
            \item $x_2 = \mathtt{c}\left( k, t \right)$ such that $k = \langle \overline{\varepsilon}, \sharp \beta, e^\prime \rangle$ for some $\beta < \alpha$, 
            and $\UCF\left( e, k \right)$ converges in at most $t$ steps. (Recall that $\overline{\varepsilon}=1-\varepsilon$, so $k$ is a $\Pi_\beta$-index if $x_1$ is a $\Sigma_\alpha$-index, and $k$ is a $\Sigma_\beta$-index if $x_1$ is a $\Pi_\alpha$-index.)
		\end{itemize}
        Then $\fun_{1}(w) := \VARsym_{1}^{k} \VARsym_{2}$. Otherwise, $\fun_{1}(w) := \VARsym_{1}$.
    \end{enumerate}
\end{definition}
It is not hard to see that both functions are computable. Obviously, they are closely related to Definition \ref{Df:infinitary_propositional_sentences}; we shall use them to simulate the inductive definition of infinitary propositional sentences within $\IAL$.

Let us describe the main construction used to prove Theorem \ref{Th:lower-bound}. Let $\VARsym_{1}$, $\VARsym_{2}$, $\VARSigma$, $\VARPi$, $\VARlt$ (``left''), $\VARrt$ (``right''), $\VARok$ (``ok''), $\VARen$ (``energy''), $\VARfn$ (``finish''), $\VARex$ (``existential''), $\VARgo$ (``go''), $\VARfl$ (``fail'') be variables. Let us regard variables of the logic $\IAL$ and symbols appearing in string rewriting systems as the same objects; i.e. we can use symbols as variables and vice versa.

\begin{definition}\label{Df:construction-hypotheses}
    For $i = 0, 1$, let $\TM_i := \left( Q_i, C_i, T_i, \qbgn_i, \qacc_i \right)$ be a Turing machine that computes the function $\fun_i$. Without loss of generality, let us assume that $Q_0 \cap Q_1 = \varnothing$. Let us define the string rewriting system $\SR_i := \SR_{\VARlt,\VARrt,\VARfn}(\TM_i)$ for $i=0,1$. (The variables $\VARlt,\VARrt$ are used as borders when modelling $\TM_i$ as $\SR_i$, and the variable $\VARfn$ indicates that a computation of $\TM_i$ is over.) Let
    \[
    \SR\ :=\
    {\SR_0 \cup \SR_1 \cup \{\VARex \to \VARsym_{2}\, \VARex, \VARex \to \VARgo\}} .
    \]
    Note that, for each rule of $\SR$, its left-hand side contains exactly one element of $Q_0 \cup Q_1 \cup \{\VARex\}$. 
    Finally, let 
    \[
    \Hc\ :=\
    {\left\{ b_1, \ldots, b_m \yields c_1 \mconj \ldots \mconj c_n \mid (b_1\ldots b_m \to c_1\ldots c_n) \in \SR \right\}} .
    \]
\end{definition}
We shall show that $\Hc$ is a desirable set of hypotheses, derivability from which is $\Sigma^0_{\omega^\omega}$-hard.

In what follows, we shall use the following notational convention for reducing parentheses in formulae: $\BS$ is left-associative and takes precedence over $\cdot$ and $\&$.

\begin{definition}\label{Df:construction-lower-bound}
Below we present the formulae used in the main construction (Lemma \ref{lemm:lower-bound-main}).
\begin{itemize}
    \item $\FOROk := \VARok \BS \VARok$.
    \item $\FORCmp(q,h) := 
		\VARgo \BS \left( q \mconj \VARrt \mconj
		\VARfn \BS \left( \left( \VARsym_{1} \BS h(\VARsym_{1}) \right) \aconj \left( \VARsym_{2} \BS h(\VARsym_{2}) \right) \right) \right)$. \\
    This formula's purpose is to initiate computation of a Turing machine with the initial state $q$ and then to analyse its results. The function $h$'s domain is $\{\VARsym_{1},\VARsym_{2}\}$; $h(\VARsym_i)$ is a formula of $\IAL$. Informally, the formula $\left( \VARsym_{1} \BS h(\VARsym_{1}) \right) \aconj \left( \VARsym_{2} \BS h(\VARsym_{2}) \right)$ says: ``if there is $\VARsym_i$ to the left, replace it by $h(\VARsym_i)$.''
    \item $\FOREn_{\exists} := 
    \VARgo \mconj 
    \FORCmp\left(\qbgn_0,h_0\right) \mconj 
    \left(\FOROk \aconj \VARgo \BS \VARex \right) \mconj
    \left(\FOROk \aconj \FORCmp\left(\qbgn_1, h_{\exists}\right) \right)$\\ where $h_0(\VARsym_{1}) = \VARok$, $h_0(\VARsym_{2}) = \VARgo$; $h_{\exists}(\VARsym_{1}) = \VARfl$, $h_{\exists}(\VARsym_{2}) = \VARPi \mconj \VARen$; 
    \item $\FOREn_{\forall} := 
		\VARgo \mconj 
        \FORCmp\left(\qbgn_0,h_0\right) \mconj 
		\left(\FOROk \aconj \VARgo \BS \left( \VARsym_{2}^\ast \mconj \VARgo \right) \right) \mconj
		\left( \FOROk \aconj \FORCmp\left(\qbgn_1, h_{\forall} \right) \right)$\linebreak
    where $h_0$ is the same function as the one defined above, and $h_{\forall}(\VARsym_{1}) = \VARok$, $h_{\forall}(\VARsym_{2}) = \VARSigma \mconj \VARen$.
    \item $\FOREn(0) := \FOROk \aconj \VARen \BS \left( \left( \VARSigma \BS \FOREn_{\exists} \right) \aconj \left( \VARPi \BS \FOREn_{\forall} \right) \right)$.
    \item $\FOREn(k+1) := \FOROk \aconj \VARen \BS (\VARen \mconj \FOREn(k)^\ast)$.
    \item $\FORBrk := \FOROk \aconj \VARen \BS \left( \left( \VARSigma \BS \VARsym_{1}^\ast \BS \VARok \right) \aconj \left( \VARPi \BS \VARsym_{1}^\ast \BS \VARok \right) \right)$.
\end{itemize}
\end{definition}
Recall that $\Upsilon_{\Hc}$ consists of formulae of the form 
\[
(b_1 \mconj \ldots \mconj b_m) \BS (c_1 \mconj \ldots \mconj c_n),
\]
where $b_1,\ldots, b_m \yields c_1 \mconj \ldots \mconj c_n$ belongs to $\Hc$.

\begin{lemm}\label{lemm:lower-bound-main}
Let $k_1 \geqslant \ldots \geqslant k_M$ be a non-increasing sequence of natural numbers, and $x = \langle \varepsilon, \sharp \alpha, e \rangle$, where $\varepsilon \in \{0,1\}$, $\alpha < \omega^\omega$, $e \in \omega$. Consider the sequent
    \begin{equation}\label{eqn:seq-main}
        {!}\Upsilon_{\Hc}, \VARlt, \VARsym_{1}^{x}, \VARsym_{\mathsf{Q}}, \VARen, \FOREn(k_M), \ldots, \FOREn(k_1), \FORBrk \yields \VARlt \mconj \VARok
    \end{equation}
where $\mathsf{Q} = \exists$ if $\varepsilon=0$ and $\mathsf{Q}=\forall$ if $\varepsilon=1$.
		\begin{enumerate}[i.]
            \item If (\ref{eqn:seq-main}) is derivable in $\EIAL$ and $\omega^{k_1}+\dotsc+\omega^{k_M} > \alpha$, then $x \in \mathbf{T}$.
            \item If $x \in \mathbf{T}$, then (\ref{eqn:seq-main}) is derivable in $\EIAL$.
		\end{enumerate}
\end{lemm}

Let us provide some intuition behind the construction of the sequent (\ref{eqn:seq-main}). Assume that one is given an index $x = \langle \varepsilon, \sharp \alpha, e \rangle$ of an infinitary propositional sentence and they want to check if it is true (i.e. whether $x \in \mathbf{T}$). The input $x$ is encoded in unary as $\VARsym_1^x$. Informally, to check whether $x \in \mathbf{T}$, we need to use the recursive definition of an infinitary propositional sentence (Definition \ref{Df:infinitary_propositional_sentences}) ``at most $\alpha$ times.'' The formulae $\FOREn_\exists$ and $\FOREn_\forall$ are used to unfold the recursive definition one time. For example, if $x$ is a $\Sigma_\alpha$-index, then $\FOREn_\exists$ ``chooses'' one of the disjuncts of the infinitary propositional sentence corresponding to $x$ and forms a sequent of the same form as (\ref{eqn:seq-main}) but where $x$ is replaced by a $\Pi_\beta$-index of the chosen disjunct (and $\VARsym_\exists$ is replaced by $\VARsym_\forall$). Similarly, the formula $\FOREn_\forall$ ``chooses'' each conjunct of the infinitary propositional sentence corresponding to $x$ and forms a new sequent with the index of the conjunct.

As one needs to unfold the inductive definition many times, the formulae $\FOREn_\forall$ and $\FOREn_\exists$ are put inside the formulae $\FOREn(k)$. ($\FOREn$ stands for \textit{energy} because $\FOREn(k)$ can be considered as a source of energy needed to solve the problem of whether $x \in \mathbf{T}$.) The formula $\FOREn(k+1)$ can produce arbitrarily many copies of $\FOREn(k)$, and hence $\FOREn(k)$ can be used to unfold the definition of an infinitary propositional sentence ``$\omega^k$ times.'' Therefore, Lemma \ref{lemm:lower-bound-main} says, if the sequent (\ref{eqn:seq-main}) contains ``sufficiently many sufficiently large'' formulae $\FOREn(k)$ (namely, such that $\omega^{k_1}+\ldots+\omega^{k_M}>\alpha$), then its derivability exactly corresponds to the fact that $x \in \mathbf{T}$.

Let us proceed with formal details. To prove the lemma we need to analyse derivability of the sequent (\ref{eqn:seq-main}) in a convenient and concise way. In Section \ref{S:dyadic}, a dyadic calculus for $\EIALm$ was presented, which simplifies analysis of the exponential modality. For our purposes, we need to further modify the latter calculus in order to normalise applications of the rules $({\BS} L)$, $({\aconj} L)$ and $({\mconj} R)$. For this sake, we now introduce, given a fixed string rewriting system $\SR$, the logic $\IALbsc$ ($\bsc$ stands for {\em basic}). Let $\SFm$ be the set of subformulae of formulae from $\{\FOREn(k) \mid k \in \omega\} \cup \{\FORBrk, \VARlt \cdot \VARok \} \cup \Var$. Observe that formulae from $\SFm$ are built using only $\mconj, \BS, \aconj, {}^*$. 

The calculus $\IALbsc$ deals with sequents $\Pi \yields C$ where $\Pi$ consists of formulae from $\SFm$ and $C \in \SFm$ equals either $r$, $r_1 \mconj r_2$ or $r^*$, where $r$, $r_i$ are variables (we call such $C$ a {\em basic right-hand side}). The axioms and rules of $\IALbsc$ are presented below (note that $r,r_i$ are variables).
\[
\infer[\ID_{\bsc}]
{r \yields r}
{}
\]
\[
\infer[({\BS} L)_{\bsc}]{\Gamma, r, r \BS B, \Delta \yields C}
{\Gamma, B, \Delta \yields C}
\qquad
\infer[({\BS} L_n)_{\bsc}\mbox{, $n \in \omega$}]{\Gamma, r^n, r^* \BS B, \Delta \yields C}
{\Gamma, B, \Delta \yields C}
\]
\[
\infer[(\mconj L)]{\Gamma, A \mconj B, \Delta \yields C}
{\Gamma, A, B, \Delta \yields C}
\qquad
\infer[(\mconj R)_{\bsc}]{r_1 , r_2 \yields r_1 \mconj r_2}
{}
\]
\[
\infer[({\aconj} L_i)_{\bsc}\mbox{, $i = 1,2$}]
{\Gamma, r_i, \left( r_1 \BS A_1 \right) \aconj \left( r_2 \BS A_2 \right), \Delta \yields C}
{\Gamma, A_i, \Delta \yields C}
\]
\[
\infer[({}^* L_\omega)]
{\Gamma, A^*, \Delta \yields C}
{\bigl( \Gamma, A^n, \Delta \yields C \bigr)_{n \in \omega}}
\qquad
\infer[({}^* R_n)_{\bsc}\mbox{, $n \in \omega$}]
{r^n \yields r^*}{}
\]
\[
\infer[(A)_{\bsc}\mbox{, provided $(b_1 \ldots b_m \to c_1 \ldots  c_n) \in \SR$}]
{\Gamma,b_1,\ldots, b_m , \Delta \yields C}
{\Gamma, c_1, \ldots, c_n, \Delta \yields C}
\]
The axioms of $\IALbsc$ are $\ID_{\bsc}$, $(\mconj R)_{\bsc}$ and $({}^* R_n)_{\bsc}$; others are rules. Note that the rules of $\IALbsc$ do not change right-hand sides of sequents and also that, for each basic right-hand side $C$, there is exactly one axiom of the form $\Pi \yields C$. If $\Gamma, \Gamma^\prime$ are sequences of formulae, then let us write $\Gamma \vdash_{\bsc} \Gamma^\prime$ if, for a fresh variable $t$ not occurring in $\Gamma,\Gamma^\prime$, there is a derivation of $\Gamma^\prime \yields t$ from $\Gamma \yields t$ in $\IALbsc$. Clearly, a sequent $\Pi \yields C$ is derivable in $\IALbsc$ if and only if there is an axiom $\Pi_0 \yields C$ of $\IALbsc$ such that $\Pi_0 \vdash_{\bsc} \Pi$. The relation $\vdash_{\bsc}$ is reflexive, transitive and monotone. Monotonicity means that if $\Pi \vdash_{\bsc} \Pi^\prime$, then $\Gamma, \Pi, \Delta \vdash_{\bsc} \Gamma, \Pi^\prime, \Delta$.

\begin{lemm}\label{lemm:lower-bound-basic-derivation}
    A sequent $\Pi \yields C$ of $\IALbsc$ is derivable in $\IALbsc$  if and only if ${!}\Upsilon_{\Hc}, \Pi \yields C$ is derivable in $\EIAL$.
\end{lemm}
\begin{proof}
    The ``only if'' direction is trivial because each rule of $\IALbsc$ is a particular case of some rule of $\EIAL$ or is a composition of those. For example, the rule $({\aconj} L_i)_{\bsc}$ is a composition of the rules $({\aconj} L_i)$ and $(\BS L)$. The rule $(A)_{\bsc}$ maps exactly onto $(A)_{\mathrm{d}}$.

    To prove the ``if'' direction, first of all, let us switch from the calculus $\EIAL$ to the dyadic calculus $\dEIALm$. By Proposition \ref{prop:dyadic-equivalence} and Lemma \ref{bd-inv}, if ${!}\Upsilon_{\Hc}, \Pi \yields C$ is derivable in $\EIAL$, then ${!}\Upsilon_{\Hc}; \Pi \yields C$ is derivable in $\dEIALm$. Take some derivation of ${!}\Upsilon_{\Hc}; \Pi \yields C$ in $\dEIALm$; we proceed with proving that then $\Pi \yields C$ is derivable in $\IALbsc$ by transfinite induction on the size of the given derivation (which is defined standardly by transfinite recursion). As for the base case, if ${!}\Upsilon_{\Hc}; \Pi \yields C$ is an axiom of $\dEIALm$, then $\Pi = C$. It is straightforward to verify that the sequents $r \yields r$, $r_1 \mconj r_2 \yields r_1 \mconj r_2$, $r^* \yields r^*$ are derivable in $\IALbsc$.

    To prove the induction step, consider the last rule application in the derivation of ${!}\Upsilon_{\Hc}; \Pi \yields C$ in $\EIALm$. Clearly, if it is that of $({\mconj} L)_{\mathrm{d}}$, $({}^* L_\omega)_{\mathrm{d}}$, or $(A)_{\mathrm{d}}$, then one simply applies the induction hypothesis to its premises and then applies the corresponding rule of $\IALbsc$ thus obtaining a correct derivation in $\IALbsc$. In particular, the rules $(A)_{\mathrm{d}}$ and $(A)_{\bsc}$ correspond to each other as they affect a sequent in the same way. Thus it is interesting to consider only the rules $({\BS} L)_{\mathrm{d}}$, $({\mconj} R)_{\mathrm{d}}$, $({\aconj} L_i)_{\mathrm{d}}$, and $({}^* R_n)_{\mathrm{d}}$ (any other rule of $\dEIALm$ cannot be the last one in the derivation of ${!}\Upsilon_{\Hc}; \Pi \yields C$ because neither $\Pi$ nor $C$ contains ${!}$, $\SL$, $\adisj$, $\Z$, $\U$).

    \textit{Case $({\BS} L)_{\mathrm{d}}$:}
    \[
        \infer[({\BS} L)_{\mathrm{d}}]
        {
            {!}\Upsilon_{\Hc}; \Gamma, \Pi, A \BS B, \Delta \yields C
        }
        {
            {!}\Upsilon_{\Hc}; \Pi \yields A
            & 
            {!}\Upsilon_{\Hc}; \Gamma, B, \Delta \yields C
        }
    \]
    If a formula of the form $A \BS B$ belongs to $\SFm$, then either $A=p$ or $A = p^*$ where $p$ is a variable (see Definition \ref{Df:construction-lower-bound}). In the first case, let us apply the induction hypothesis to the premises and conclude that $\Pi \yields p$ and $\Gamma, B, \Delta \yields C$ are derivable in $\IALbsc$, i.e. $p \vdash_{\bsc} \Pi$ and $\Pi_0 \vdash_{\bsc} \Gamma, B, \Delta$ where $\Pi_0 \yields C$ is an axiom of $\IALbsc$. Therefore, $\Pi_0 \vdash_{\bsc} \Gamma, B, \Delta \vdash_{\bsc} \Gamma, p, p \BS B, \Delta \vdash_{\bsc} \Gamma, \Pi, p \BS B, \Delta$; thus $\Gamma, \Pi, p \BS B, \Delta \yields C$ is derivable in $\IALbsc$ as desired. In the case $A = p^*$, the reasonings are similar, the only difference is that $p^n \vdash_{\bsc} \Pi$ for some $n$.

    \textit{Case $({\mconj} R)_{\mathrm{d}}$:}
    \[
        \infer[(\mconj R)_{\mathrm{d}}]
        {{!}\Upsilon_{\Hc}; \Gamma, \Delta \yields r_1 \mconj r_2}
        {{!}\Upsilon_{\Hc}; \Gamma \yields r_1 & {!}\Upsilon_{\Hc}; \Delta \yields r_2}
    \]
    Similarly to the previous case, one applies the induction hypothesis to the premises and obtains that $r_1 \vdash_{\bsc} \Gamma$, $r_2 \vdash_{\bsc} \Delta$. Thus, $r_1, r_2 \vdash_{\bsc} \Gamma, r_2 \vdash_{\bsc} \Gamma, \Delta$ and hence $\Gamma, \Delta \yields r_1 \mconj r_2$ is derivable from the axiom $r_1,r_2 \yields r_1 \mconj r_2$.

    \textit{Case $({}^* R_n)_{\mathrm{d}}$:}
    \[
        \infer[({}^* R_n)_{\mathrm{d}}]
            {{!}\Upsilon_{\Hc}; \Pi_1, \ldots, \Pi_n \yields r^*}
            {{!}\Upsilon_{\Hc}; \Pi_1 \yields r & \ldots & {!}\Upsilon_{\Hc}; \Pi_n \yields r}
    \]
    The argument is similar; namely, using the induction hypothesis one concludes that $r \vdash_{\bsc} \Pi_i$ for $i=1,\ldots, n$ and hence $r^n \vdash_{\bsc} \Pi_1, r^{n-1} \vdash_{\bsc} \ldots \vdash_{\bsc} \Pi_1, \ldots, \Pi_n$. Since $r^n \yields r^\ast$ is an axiom of $\IALbsc$, the sequent $\Pi_1, \ldots, \Pi_n \yields r^\ast$ is derivable in $\IALbsc$.

    \textit{Case $({\aconj} L_i)_{\mathrm{d}}$:}
    \[
        \infer[({\aconj} L_i)_{\mathrm{d}}]
        {{!}\Upsilon_{\Hc}; \Gamma, \left( r_1 \BS A_1 \right) \aconj \left( r_2 \BS A_2 \right), \Delta \yields C}
        {{!}\Upsilon_{\Hc}; \Gamma, r_i \BS A_i, \Delta \yields C}
    \]
    (Note that any formula $B_1 \aconj B_2 \in \SFm$ is of the form $\left( r_1 \BS A_1 \right) \aconj \left( r_2 \BS A_2 \right)$ for some variables $r_1,r_2$.) By the induction hypothesis, there is a derivation of $\Gamma, r_i \BS A_i, \Delta \yields C$ in $\IALbsc$. Then, one reasons as follows. Given some fixed derivation of $\Gamma, r_i \BS A_i, \Delta \yields C$ in $\IALbsc$, find all the occurrences of the formula $r_i \BS A_i$ and replace them by $\left( r_1 \BS A_1 \right) \aconj \left( r_2 \BS A_2 \right)$. Then, each application of $({\BS} L)_{\bsc}$ that introduces $r_i \BS A_i$ turns into an application of $({\aconj} L_i)_{\bsc}$ that introduces $\left( r_1 \BS A_1 \right) \aconj \left( r_2 \BS A_2 \right)$. These arguments can be formalised easily by means of induction.
\end{proof}

\begin{coro}\label{coro:lower-bound-basic-derivation}
    Consider a sequent $\Gamma, B, \Delta \yields C$ of $\IALbsc$ where $\Gamma$ is a sequence of variables and $\Delta$ is a sequence of formulae of the form $\FOROk \aconj \left( q \BS C \right)$ where $q \in \Var$. 
    \begin{enumerate}[i.]
        \item If $B = p \BS A$ for $p \in \Var$, then the sequent is derivable if and only if there is $\Gamma^\prime$ such that $\Gamma \Rightarrow_{\SR}^\ast \Gamma^\prime\, p$ and $\Gamma^\prime, A, \Delta \yields C$ is derivable in $\IALbsc$.
        \item If $B = \FOROk \aconj (p \BS A)$ and $\Gamma$ does not contain $\VARok$, then the sequent is derivable if and only if there is $\Gamma^\prime$ such that $\Gamma \Rightarrow_{\SR}^\ast \Gamma^\prime\, p$ and $\Gamma^\prime, A, \Delta \yields C$ is derivable in $\IALbsc$.
    \end{enumerate}
    
\end{coro}
\begin{proof}
    Let us prove the first statement. The ``if'' direction is straightforward because $\Gamma, p \BS A, \Delta \yields C$ can be obtained from $\Gamma^\prime, A, \Delta \yields C$ by an application of $(\BS L)_{\bsc}$ followed by a series of applications of $(A)_{\bsc}$ that correspond to the derivation $\Gamma \Rightarrow_{\SR}^\ast \Gamma^\prime\, p$. 
    
    The ``only if'' direction is proved by induction on the size of a derivation of $\Gamma, p \BS A, \Delta \yields C$ in $\IALbsc$. The last rule applied in it is either $({\BS} L)_{\bsc}$ or $(A)_{\bsc}$. In the first case, $\Gamma = \Gamma^\prime, p$, and the rule application looks as follows:
    \[
        \infer[({\BS} L)_{\bsc}]{\Gamma^\prime, p, p \BS A, \Delta \yields C}{\Gamma^\prime, A, \Delta \yields C}
    \]
    The statement of the corollary follows. In the second case, $\Gamma = \Gamma_1, b_1, \ldots, b_m, \Gamma_2$ and the rule looks as follows:
    \[
        \infer[(A)_{\bsc}\mbox{, $(b_1 \ldots b_m \to c_1 \ldots  c_n) \in \SR$}]
        {\Gamma_1,b_1,\ldots, b_m , \Gamma_2, p \BS A, \Delta \yields C}
        {\Gamma_1, c_1, \ldots, c_n, \Gamma_2, p \BS A, \Delta  \yields C}
    \]
    By the induction hypothesis, there is $\Gamma^\prime$ such that $\Gamma_1, c_1, \ldots, c_n, \Gamma_2 \Rightarrow_{\SR}^\ast \Gamma^\prime\, p$ and $\Gamma^\prime, A, \Delta \yields C$ is derivable in $\IALbsc$. It remains to observe that $\Gamma = \Gamma_1,b_1,\ldots, b_m , \Gamma_2 \Rightarrow_{\SR} \Gamma_1, c_1, \ldots, c_n, \Gamma_2 \Rightarrow_{\SR}^\ast \Gamma^\prime\, p$.

    The second statement is proved in the same way. Notice that no rewriting rule of $\SR$ contains $\VARok$, so $\VARok$ cannot appear to the left from $\FOROk \aconj (p \BS A)$.
\end{proof}

\begin{lemm}\label{lemm:ok}
    A sequent of the form $\VARlt, \VARok, \FOROk \aconj B_1, \ldots, \FOROk \aconj B_n \yields \VARlt \mconj \VARok$ is derivable in $\IALbsc$.
\end{lemm}
\begin{proof}
    Recall that $\FOROk = \VARok \BS \VARok$. Then the sequent can be derived in $\IALbsc$ from the axiom $\VARlt , \VARok \yields \VARlt \mconj \VARok$ by $n$ applications of $({\aconj} L_1)_{\bsc}$.
\end{proof}

We are finally ready to prove the main lemma.
\begin{proof}[Proof of Lemma \ref{lemm:lower-bound-main}]
    The proof is by transfinite induction on $\gamma := \omega^{k_1}+\dotsc+\omega^{k_M}$ (if $M=0$, then $\gamma := 0$). According to Lemma \ref{lemm:lower-bound-basic-derivation}, the sequent (\ref{eqn:seq-main}) is derivable in $\EIAL$ if and only if the following sequent is derivable in $\IALbsc$:
    \begin{equation}\label{eqn:seq-main-bsc}
		\VARlt, \VARsym_{1}^{x}, \VARsym_{\mathsf{Q}}, \VARen, \FOREn(k_M), \ldots, \FOREn(k_1), \FORBrk \yields \VARlt \mconj \VARok
    \end{equation}

    \textit{The base case:} $\gamma = 0$. Then, (\ref{eqn:seq-main-bsc}) is of the form $\VARlt, \VARsym_{1}^{x}, \VARsym_{\mathsf{Q}}, \VARen, \FORBrk \yields \VARlt \mconj \VARok$. This sequent is derivable in $\IALbsc$; below, we present a derivation for $\mathsf{Q} = \exists$:
    \[
        \infer[({\aconj} L_2)_{\bsc}]
		{
            \VARlt, \VARsym_{1}^{x}, \VARSigma, \VARen, \FOROk \aconj \VARen \BS \left( \left( \VARSigma \BS \VARsym_{1}^\ast \BS \VARok \right) \aconj \left( \VARPi \BS \VARsym_{1}^\ast \BS \VARok \right) \right) \yields \VARlt \mconj \VARok
		}
		{
			\infer[({\aconj} L_1)_{\bsc}]
				{
				    \VARlt, \VARsym_{1}^{x}, \VARSigma,  \left( \VARSigma \BS \VARsym_{1}^\ast \BS \VARok \right) \aconj \left( \VARPi \BS \VARsym_{1}^\ast \BS \VARok \right) \yields \VARlt \mconj \VARok
				}
				{
					\infer[({\BS} L_x)_{\bsc}]
						{
						  \VARlt, \VARsym_{1}^{x}, \VARsym_{1}^\ast \BS \VARok \yields \VARlt \mconj \VARok
						}
						{
                            \VARlt, \VARok \yields \VARlt \mconj \VARok
						}
				}
		}
    \]
    This agrees with the statement of Lemma \ref{lemm:lower-bound-main}. Indeed, since it cannot be the case that $0 = \gamma > \alpha$ for any ordinal $\alpha$, the first statement of the lemma holds trivially. The second one holds as well because the sequent is derivable independently of~$x$.

    \textit{The inductive step:} assume that $\gamma > 0$ (or equivalently, that $M > 0$).
    
    \textit{Case 1:} $k_M > 0$. Recall that $\FOREn(k_M) := \FOROk \aconj \VARen \BS (\VARen \mconj \FOREn(k_M-1)^\ast)$. Then, any derivation of (\ref{eqn:seq-main-bsc}) in the calculus $\IALbsc$ must end as follows:
    \begin{equation}\label{eqn:der-base-case}
    \infer[(\aconj L_2)_{\bsc}]
    {
        \VARlt, \VARsym_{1}^{x}, \VARsym_{\mathsf{Q}}, \VARen, \FOREn(k_M), \ldots, \FOREn(k_1), \FORBrk \yields \VARlt \mconj \VARok
    }
    {
        \infer[(\mconj L)]
        {
            \VARlt, \VARsym_{1}^{x}, \VARsym_{\mathsf{Q}}, \VARen \mconj \FOREn(k_M-1)^\ast, \FOREn(k_{M-1}), \ldots, \FOREn(k_1), \FORBrk \yields \VARlt \mconj \VARok
        }
        {
            \infer[({}^* L_\omega)]
            {
                \VARlt, \VARsym_{1}^{x}, \VARsym_{\mathsf{Q}}, \VARen , \FOREn(k_M-1)^\ast, \FOREn(k_{M-1}), \ldots, \FOREn(k_1), \FORBrk \yields \VARlt \mconj \VARok
            }
            {
                \left(
                    \VARlt, \VARsym_{1}^{x}, \VARsym_{\mathsf{Q}}, \VARen , \FOREn(k_M-1)^{l}, \FOREn(k_{M-1}), \ldots, \FOREn(k_1), \FORBrk \yields \VARlt \mconj \VARok
                \right)_{l \in \omega}
            }
        }
    }
    \end{equation}
    The fact that there are no other ways to derive~(\ref{eqn:seq-main-bsc}) is verified straightforwardly by considering all possible ways to apply rules of $\IALbsc$. In particular, there is no way to apply the rule $(A)_{\bsc}$ at any of the above steps because this is possible only if there is a variable from $Q_0 \cup Q_1 \cup \{\VARex\}$ in the left-hand side of a sequent.
    
    Now, we can apply the induction hypothesis to the sequent
    \begin{equation}\label{eqn:IH-base-case}
        \VARlt, \VARsym_{1}^{x}, \VARsym_{\mathsf{Q}}, \VARen , \FOREn(k_M-1)^{l}, \FOREn(k_{M-1}), \ldots, \FOREn(k_1), \FORBrk \yields \VARlt \mconj \VARok
    \end{equation}
    for each $l$. Indeed, the parameter $\gamma_l := \omega^{k_1}+\dotsc+\omega^{k_{M-1}}+\omega^{k_M-1}\cdot l$ of this sequent is less than $\gamma$. By the induction hypothesis, if (\ref{eqn:IH-base-case}) is derivable and $\gamma_l > \alpha$, then $x \in \mathbf{T}$, and, conversely, if $x \in \mathbf{T}$, then (\ref{eqn:IH-base-case}) is derivable. 
    
    Given this, let us prove both statements of the lemma for the sequent (\ref{eqn:seq-main-bsc}). Assume that (\ref{eqn:seq-main-bsc}) is derivable and $\gamma > \alpha$. Then (\ref{eqn:IH-base-case}) is derivable for each $l \in \omega$ (since this sequent occurs in the only possible derivation (\ref{eqn:der-base-case})). Since $\gamma > \alpha$ and $\gamma = \sup_{l \in \omega} \gamma_l$, then $\gamma_l > \alpha$ for some $l \in \omega$. Therefore, by the induction hypothesis, $x \in \mathbf{T}$, as desired. 
    Conversely, let $x \in \mathbf{T}$. Then, (\ref{eqn:IH-base-case}) is derivable for each $l \in \omega$ and hence so is (\ref{eqn:seq-main-bsc}), as the derivation (\ref{eqn:der-base-case}) shows.
    
    \textit{Case 2:} $k_M = 0$. Recall that $\FOREn(0) := \FOROk \aconj \VARen \BS \left( \left( \VARSigma \BS \FOREn_{\exists} \right) \aconj \left( \VARPi \BS \FOREn_{\forall} \right) \right)$. Let us denote the sequence of formulae $\FOREn(k_{M-1}), \ldots, \FOREn(k_1), \FORBrk$ by $\Xi$. Any derivation of (\ref{eqn:seq-main-bsc}) must end as follows:
    \[
        \infer[(\aconj L_2)_{\bsc}]
        {
            \VARlt, \VARsym_{1}^{x}, \VARsym_{\mathsf{Q}}, \VARen, \FOREn(0), \Xi \yields \VARlt \mconj \VARok
        }
        {
            \infer[(\aconj L_i)_{\bsc}]
            {
                \VARlt, \VARsym_{1}^{x}, \VARsym_{\mathsf{Q}}, \left( \VARSigma \BS \FOREn_{\exists} \right) \aconj \left( \VARPi \BS \FOREn_{\forall} \right), \Xi \yields \VARlt \mconj \VARok
            }
            {
                \VARlt, \VARsym_{1}^{x}, \FOREn_{\mathsf{Q}}, \Xi \yields \VARlt \mconj \VARok
            }
        }
    \]
    Here $i=1$ if $\mathsf{Q}=\exists$ and $i=2$ if $\mathsf{Q} = \forall$.
    
    \textit{Case 2a:} $\mathsf{Q} = \exists$. Recall that $\FOREn_{\exists} = 
    \VARgo \mconj 
    \FORCmp\left(\qbgn_0,h_0\right) \mconj 
    \left(\FOROk \aconj \VARgo \BS \VARex \right) \mconj
    \left(\FOROk \aconj \FORCmp\left(\qbgn_1, h_{\exists}\right) \right)$. Then, by Corollary \ref{coro:invertibility}, $\VARlt, \VARsym_{1}^{x}, \FOREn_{\exists}, \Xi \yields \VARlt \mconj \VARok$ is derivable if and only if so is
    \[
        \VARlt, \VARsym_{1}^{x}, \VARgo, \FORCmp\left(\qbgn_0,h_0\right), \FOROk \aconj \VARgo \BS \VARex, \FOROk \aconj \FORCmp\left(\qbgn_1, h_{\exists}\right), \Xi \yields \VARlt \mconj \VARok.
    \]
    Recall that $\FORCmp\left(\qbgn_0,h_0\right) = \VARgo \BS \left( \qbgn_0 \mconj \VARrt \mconj \VARfn \BS \left( \left( \VARsym_{1} \BS h_0(\VARsym_{1}) \right) \aconj \left( \VARsym_{2} \BS h_0(\VARsym_{2}) \right) \right) \right)$. The only way to derive the latter sequent in $\IALbsc$ is by means of the following application of $({\BS} L)_{\bsc}$:
    \[
    \infer %[({\BS} L)_{\bsc}]
    {
        \VARlt, \VARsym_{1}^{x}, \VARgo, \FORCmp\left(\qbgn_0,h_0\right), \FOROk \aconj \VARgo \BS \VARex, \FOROk \aconj \FORCmp\left(\qbgn_1, h_{\exists}\right), \Xi \yields \VARlt \mconj \VARok
    }
    {
    \parbox{.9\textwidth}{
        $\VARlt, \VARsym_{1}^{x}, \qbgn_0 \mconj \VARrt \mconj \VARfn \BS \left( \VARsym_{1} \BS h_0(\VARsym_{1}) \aconj \VARsym_{2} \BS h_0(\VARsym_{2}) \right), \\
        \phantom{a} \hfill \FOROk \aconj \VARgo \BS \VARex, \FOROk \aconj \FORCmp\left(\qbgn_1, h_{\exists}\right), \Xi \yields \VARlt \mconj \VARok$
    }
    }
    \]
    Again, by Corollary \ref{coro:invertibility}, the premise is equiderivable with the sequent
    \begin{equation}\label{eqn:2a}
        \VARlt, \VARsym_{1}^{x}, \qbgn_0 , \VARrt , \VARfn \BS \left( \VARsym_{1} \BS h_0(\VARsym_{1}) \aconj \VARsym_{2} \BS h_0(\VARsym_{2}) \right), \Xi^\prime \yields \VARlt \mconj \VARok.
    \end{equation}
    where $\Xi^\prime = \FOROk \aconj \VARgo \BS \VARex, \FOROk \aconj \FORCmp\left(\qbgn_1, h_{\exists}\right), \Xi$. The sequent (\ref{eqn:2a}) cannot be obtained by any rule application except for that of $(A)_{\bsc}$. By Corollary \ref{coro:lower-bound-basic-derivation}, it is derivable if and only if $\VARlt\, \VARsym_{1}^{x}\, \qbgn_0 \, \VARrt \Rightarrow_{\SR}^\ast \Gamma^\prime\, \VARfn$ for $\Gamma^\prime$ such that $\Gamma^\prime, \VARsym_{1} \BS h_0(\VARsym_{1}) \aconj \VARsym_{2} \BS h_0(\VARsym_{2}), \Xi^\prime \yields \VARlt \mconj \VARok$ is derivable. Note that, if $\VARlt\, \VARsym_{1}^{x}\, \qbgn_0 \, \VARrt \Rightarrow_{\SR}^\ast \Gamma^\prime\, \VARfn$, then $\VARlt\, \VARsym_{1}^{x}\, \qbgn_0 \, \VARrt \Rightarrow_{\SR_0}^\ast \Gamma^\prime\, \VARfn$ because the initial string contains a state from $Q_0$ but it contains neither states from $Q_1$ nor $\VARex$. This property is preserved by any rule application of $\SR$, so, in the derivation $\VARlt\, \VARsym_{1}^{x}\, \qbgn_0 \, \VARrt \Rightarrow_{\SR}^\ast \Gamma^\prime\, \VARfn$, only rules from $\SR_0$ can be used.
    
    Now, it is time to recall some properties of $\SR_0 = \SR_{\VARlt,\VARrt,\VARfn}\left(\TM_0 \right)$. By Proposition \ref{prop:Church-Turing}, $\VARlt\, \VARsym_{1}^{x}\, \qbgn_0 \, \VARrt \Rightarrow_{\SR_0}^\ast \Gamma^\prime\, \VARfn$ holds if and only if $\Gamma^\prime = \VARlt \, \fun_{0}(\VARsym_{1}^{x})$. Therefore, (\ref{eqn:2a}) is derivable if and only if so is
    \begin{equation}\label{eqn:2a'}
        \VARlt , \fun_{0}(\VARsym_{1}^{x}), \VARsym_{1} \BS h_0(\VARsym_{1}) \aconj \VARsym_{2} \BS h_0(\VARsym_{2}), \Xi^\prime \yields \VARlt \mconj \VARok.
    \end{equation}
    Note that $\fun_{0}\left(\VARsym_{1}^{x}\right)$ ends with either $\VARsym_{1}$ or $\VARsym_{2}$; let us write $\fun_{0}\left(\VARsym_{1}^{x}\right) = \Delta\, \VARsym_i$. The sequent (\ref{eqn:2a'}) does not contain symbols from $Q_0 \cup Q_1 \cup \{\VARex\}$, so it cannot be obtained by $(A)_{\bsc}$; hence it can only be obtained by $({\aconj} L_i)_{\bsc}$ as follows:
    \[
        \infer[({\aconj} L_i)_{\bsc}]
        {
            \VARlt , \Delta, \VARsym_i, \VARsym_{1} \BS h_0(\VARsym_{1}) \aconj \VARsym_{2} \BS h_0(\VARsym_{2}), \Xi^\prime \yields \VARlt \mconj \VARok
        }
        {
            \VARlt , \Delta, h_0(\VARsym_i), \Xi^\prime \yields \VARlt \mconj \VARok
        }
    \]
    Let us consider two subcases.
    
    \textit{Subcase i.} $\fun_{0}\left(\VARsym_{1}^{x}\right) = \VARsym_{1}$. This happens, according to Definition \ref{Df:check-functions}, if and only if $\alpha = 0$ (recall that $x = \langle \varepsilon, \sharp \alpha, e \rangle$) and $x \in \mathbf{T}$. Since $h_0(\VARsym_1)=\VARok$, the sequent $\VARlt , \Delta, h_0(\VARsym_{1}), \Xi^\prime \yields \VARlt \mconj \VARok$ equals $\VARlt , \VARok, \FOROk \aconj \VARgo \BS \VARex, \FOROk \aconj \FORCmp\left(\qbgn_1, h_{\exists}\right), \Xi \yields \VARlt \mconj \VARok$ and is derivable by Lemma \ref{lemm:ok}.
		
    \textit{Subcase ii.} $\fun_{0}\left(\VARsym_{1}^{x}\right) = \VARsym_{1}^{x} \VARsym_{2}$. This happens if and only if $\alpha > 0$. The sequent $\VARlt , \Delta, h_0(\VARsym_{2}), \Xi^\prime \yields \VARlt \mconj \VARok$ equals 
    \begin{equation}\label{eqn:2aii}
        \VARlt, \VARsym_{1}^{x}, \VARgo, \FOROk \aconj \VARgo \BS \VARex, \FOROk \aconj \FORCmp\left(\qbgn_1, h_{\exists}\right), \Xi \yields \VARlt \mconj \VARok.
    \end{equation}
    	
    Concluding the arguments above, if $\alpha = 0$, then the sequent (\ref{eqn:seq-main-bsc}) we started with is derivable in $\IALbsc$ if and only if $x \in \mathbf{T}$. If $\alpha > 0$, then derivability of (\ref{eqn:seq-main-bsc}) is equivalent to that of (\ref{eqn:2aii}). The latter one's derivation in $\IALbsc$ must end as follows: 
    \[
        \infer[({\aconj} L_2)_{\bsc}]
        {
            \VARlt, \VARsym_{1}^{x}, \VARgo, \FOROk \aconj \VARgo \BS \VARex, \FOROk \aconj \FORCmp\left(\qbgn_1, h_{\exists}\right), \Xi \yields \VARlt \mconj \VARok
        }
        {
            \VARlt, \VARsym_{1}^{x}, \VARex, \FOROk \aconj \FORCmp\left(\qbgn_1, h_{\exists}\right), \Xi \yields \VARlt \mconj \VARok
        }
    \]
    The premise can be obtained only by means of the rule $(A)_{\bsc}$, which can indeed be applied because there is an occurrence of $\VARex$ in the left-hand side. Recall that
    \[
        \FORCmp\left(\qbgn_1,h_{\exists}\right) = \VARgo \BS \left( \qbgn_1 \mconj \VARrt \mconj \VARfn \BS \left( \left( \VARsym_{1} \BS h_{\exists}(\VARsym_{1}) \right) \aconj \left( \VARsym_{2} \BS h_{\exists}(\VARsym_{2}) \right) \right) \right)
    \]
    By Corollary \ref{coro:lower-bound-basic-derivation}, there exists $\Gamma^{\prime\prime}$ such that $\VARlt \, \VARsym_{1}^{x} \, \VARex \Rightarrow_{\SR}^\ast \Gamma^{\prime\prime} \, \VARgo$ and the sequent
    \begin{equation}\label{eqn:2a''}
        \Gamma^{\prime\prime}, \qbgn_1 \mconj \VARrt \mconj \VARfn \BS \left( \left( \VARsym_{1} \BS h_{\exists}(\VARsym_{1}) \right) \aconj \left( \VARsym_{2} \BS h_{\exists}(\VARsym_{2}) \right) \right), \Xi \yields \VARlt \mconj \VARok
    \end{equation}
    is derivable. The string rewriting system $\SR$ contains rules from $\SR_0$ and $\SR_1$ and also the rules $\VARex \to \VARsym_{2}\, \VARex, \VARex \to \VARgo$; note that each rule from $\SR_0 \cup \SR_1$ contains an element of $Q_0 \cup Q_1$ in left-hand sides. Since $\VARlt \, \VARsym_{1}^{x} \, \VARex$ does not contain symbols from $Q_0 \cup Q_1$, none of these rules can be applied in a derivation $\VARlt \, \VARsym_{1}^{x} \, \VARex \Rightarrow_{\SR}^\ast \Gamma^{\prime\prime} \, \VARgo$. Thus, $\Gamma^{\prime\prime} \, \VARgo$ is obtained from $\VARlt \, \VARsym_{1}^{x} \, \VARex$ by means of the rules $\VARex \to \VARsym_{2} \, \VARex$, $\VARex \to \VARgo$. Therefore, there is $y \in \omega$ such that $\Gamma^{\prime\prime} = \VARlt \, \VARsym_{1}^{x} \, \VARsym_{2}^{y}$. 

    By Corollary \ref{coro:invertibility}, the sequent (\ref{eqn:2a''}) is equiderivable with the following
    \begin{equation}\label{eqn:2a'''}
        \VARlt, \VARsym_{1}^{x}, \VARsym_{2}^{y}, \qbgn_1 , \VARrt , \VARfn \BS \left( \left( \VARsym_{1} \BS h_{\exists}(\VARsym_{1}) \right) \aconj \left( \VARsym_{2} \BS h_{\exists}(\VARsym_{2}) \right) \right), \Xi \yields \VARlt \mconj \VARok.
    \end{equation}
    Derivability of this sequent is analysed in the same way as that of (\ref{eqn:2a}). Namely, (\ref{eqn:2a'''}) is derivable if and only if so is 
    \begin{equation}\label{eqn:2a''''}
        \VARlt , \Delta^\prime, h_{\exists}(\VARsym_i), \Xi \yields \VARlt \mconj \VARok
    \end{equation}
    for $\Delta^\prime \, \VARsym_i = \fun_{1}\left(\VARsym_{1}^{x} \VARsym_{2}^{y}\right)$. Again, two subcases arise.
    
    \textit{Subcase A.} $\fun_{1}\left(\VARsym_{1}^{x} \VARsym_{2}^{y}\right) = \VARsym_{1}^{k} \VARsym_{2}$. This happens if and only if $y = \mathtt{c}\left( k, t \right)$ such that $k = \langle \overline{\varepsilon}, \sharp \beta, e^\prime \rangle$ for some $\beta < \alpha$ and $e^\prime \in \omega$ and such that $\UCF\left( e, k \right)$ converges in at most $t$ steps. In this case, $\Delta^\prime = \VARsym_{1}^{k}$ and the sequent (\ref{eqn:2a''''}) has the form $\VARlt , \VARsym_{1}^{k}, \VARsym_\forall \mconj \VARen, \FOREn(k_{M-1}), \ldots, \FOREn(k_1), \FORBrk \yields \VARlt \mconj \VARok$ (recall that $h_\exists(\VARsym_2) = \VARsym_\forall \mconj \VARen$). Its derivability is equivalent to that of
    \begin{equation}\label{eqn:2aA}
        \VARlt , \VARsym_{1}^{k}, \VARsym_\forall , \VARen, \FOREn(k_{M-1}), \ldots, \FOREn(k_1), \FORBrk \yields \VARlt \mconj \VARok
    \end{equation}
		
    \textit{Subcase B.} $\fun_{1}\left(\VARsym_{1}^{x} \VARsym_{2}^{y}\right) = \VARsym_{1}$. In this case, the sequent (\ref{eqn:2a''''}) has the form 
    \[
    \VARlt , \VARfl, \FOREn(k_{M-1}), \ldots, \FOREn(k_1), \FORBrk \yields \VARlt \mconj \VARok
    \]
    because $h_\exists(\VARsym_1) = \VARfl$. It is not derivable in $\IALbsc$, as one cannot obtain this sequent by means of any rule of $\IALbsc$. (Note that $\VARfl$ is a special variable that does not occur elsewhere; its purpose is to indicate that the computation fails because of the invalid input $(x,y)$.)

    Let us summarize all the reasonings for Case 2a.
    \begin{itemize}
        \item If (\ref{eqn:seq-main-bsc}) is derivable and $\gamma = \omega^{k_1} + \ldots + \omega^{k_M} > \alpha$, then the sequent (\ref{eqn:2aA}) is derivable for some $k = \langle \overline{\varepsilon}, \sharp \beta, e^\prime \rangle$ such that $\beta < \alpha$ and $\UCF\left( e, k \right)$ converges. The parameter $\gamma^\prime$ of the sequent (\ref{eqn:2aA}) is $\omega^{k_1}+\ldots+\omega^{k_{M-1}}$, and it is greater than $\beta$. The induction hypothesis implies that $k \in \mathbf{T}$. The latter implies that $x \in \mathbf{T}$ because $x$ is an index of a disjunction of formulae such that one of the disjuncts has index $k$.
        \item If $x$ belongs to $\mathbf{T}$, then there exists $k = \langle \overline{\varepsilon}, \sharp \beta, e^\prime \rangle \in \mathbf{T}$ such that $\beta < \alpha$ and $\UCF\left( e, k \right)$ converges. Let $t$ be the number of steps after which $\UCF\left( e, k \right)$ converges. One can apply the induction hypothesis and conclude that (\ref{eqn:2aA}) is derivable. Then, the sequent (\ref{eqn:2a'''}) is derivable for $y = \mathtt{c}\left( e, t \right)$, which implies the derivability of (\ref{eqn:seq-main-bsc}).
    \end{itemize}
    This concludes the proof of Case 2a.

    \textit{Case 2b:} $X = \Pi$. This case is dealt with in the same manner as Case 2a because $\FOREn_{\exists}$ and $\FOREn_{\forall}$ have similar structure. Let us briefly outline the main stages of the argument. Recall that $\FOREn_{\forall} = \VARgo \mconj 
        \FORCmp\left(\qbgn_0,h_0\right) \mconj 
		\left(\FOROk \aconj \VARgo \BS \left( \VARsym_{2}^\ast \mconj \VARgo \right) \right) \mconj
		\left( \FOROk \aconj \FORCmp\left(\qbgn_1, h_{\forall} \right) \right)$. First, one shows that the sequent $\VARlt, \VARsym_{1}^{x}, \FOREn_{\forall}, \Xi \yields \VARlt \mconj \VARok$ is derivable if and only if one of the following holds:
    \begin{enumerate}[i.]
        \item $\alpha=0$ and $x \in \mathbf{T}$;
        \item $\alpha>0$ and the following sequent is derivable in $\IALbsc$:
        \begin{equation}\label{eqn:2b}
            \VARlt, \VARsym_{1}^{x}, \VARgo, \FOROk \aconj \VARgo \BS \left( \VARsym_{2}^\ast \mconj \VARgo \right) , \FOROk \aconj \FORCmp\left(\qbgn_1, h_{\forall} \right) , \Xi \yields \VARlt \mconj \VARok.
        \end{equation}
        (Compare it with the sequent (\ref{eqn:2aii}).)
    \end{enumerate}
    Any derivation of the sequent (\ref{eqn:2b}) must end as follows:
    \[
        \infer[({\aconj} L_2)_{\bsc}]
        {
            \VARlt, \VARsym_{1}^{x}, \VARgo, \FOROk \aconj \VARgo \BS \left( \VARsym_{2}^\ast \mconj \VARgo \right) , \FOROk \aconj \FORCmp\left(\qbgn_1, h_{\forall} \right) , \Xi \yields \VARlt \mconj \VARok
        }
        {
            \infer[(\mconj L)]
            {
                \VARlt, \VARsym_{1}^{x}, \VARsym_{2}^\ast \mconj \VARgo , \FOROk \aconj \FORCmp\left(\qbgn_1, h_{\forall} \right) , \Xi \yields \VARlt \mconj \VARok
            }
            {
                    \VARlt, \VARsym_{1}^{x}, \VARsym_{2}^\ast , \VARgo , \FOROk \aconj \FORCmp\left(\qbgn_1, h_{\forall} \right) , \Xi \yields \VARlt \mconj \VARok
            }
        }
    \]
    It follows from Corollary \ref{coro:invertibility} that the sequent 
    \begin{equation}\label{eqn:2b*}
        \VARlt, \VARsym_{1}^{x}, \VARsym_{2}^\ast , \VARgo , \FOROk \aconj \FORCmp\left(\qbgn_1, h_{\forall} \right) , \Xi \yields \VARlt \mconj \VARok
    \end{equation}
    is derivable if and only if so is 
    \begin{equation}\label{eqn:2b'}
        \VARlt, \VARsym_{1}^{x}, \VARsym_{2}^y , \VARgo , \FOROk \aconj \FORCmp\left(\qbgn_1, h_{\forall} \right) , \Xi \yields \VARlt \mconj \VARok
    \end{equation}
    for each $y \in \omega$. The sequent (\ref{eqn:2b'}) is equiderivable with
    \begin{equation}\label{eqn:2b''}
        \VARlt, \VARsym_{1}^{x}, \VARsym_{2}^{y}, \qbgn_1 , \VARrt , \VARfn \BS \left( \left( \VARsym_{1} \BS h_{\forall}(\VARsym_{1}) \right) \aconj \left( \VARsym_{2} \BS h_{\forall}(\VARsym_{2}) \right) \right), \Xi \yields \VARlt \mconj \VARok.
    \end{equation}
    This one is similar to (\ref{eqn:2a'''}), so again two subcases arise.
 
 If $\fun_{1}(\VARsym_{1}^x \VARsym_{2}^y) = \VARsym_{1}^{k} \VARsym_{2}$, then $y = \mathtt{c}\left( k, t \right)$, where $k = \langle \overline{\varepsilon}, \sharp \beta, e^\prime \rangle$ for $\beta < \alpha$ and $\UCF\left( e, k \right)$ converges. Let $t$ be an upper bound on the number of steps in which $\UCF\left( e, k \right)$ converges. In this case, the sequent (\ref{eqn:2b''}) is derivable if and only if so is
        \begin{equation}\label{eqn:2bA}
            \VARlt , \VARsym_{1}^{k}, \VARsym_\exists , \VARen, \FOREn(k_{M-1}), \ldots, \FOREn(k_1), \FORBrk \yields \VARlt \mconj \VARok
        \end{equation}
     
      Otherwise, $\fun_{1}\left(\VARsym_{1}^{x} \VARsym_{2}^{y}\right) = \VARsym_{1}$. In this case, the sequent (\ref{eqn:2b''}) is equiderivable with the sequent 
        \[
        \VARlt , \VARok, \FOREn(k_{M-1}), \ldots, \FOREn(k_1), \FORBrk \yields \VARlt \mconj \VARok
        \] 
        (recall that $h_{\forall}(\VARsym_{1}) = \VARok$). It is derivable in $\IALbsc$ by Lemma \ref{lemm:ok}.

    Combining these arguments we obtain the following.

\begin{itemize}
\item
    If (\ref{eqn:seq-main-bsc}) is derivable and $\omega^{k_1}+\ldots+\omega^{k_M}>\alpha$, then the sequent (\ref{eqn:2bA}) is derivable for each $k = \langle \overline{\varepsilon}, \sharp \beta, e^\prime \rangle$ such that $\beta < \alpha$, $e^\prime \in \omega$ and $\UCF\left( e, k \right)$ converges. By the induction hypothesis, this means that all such $k$'s are members of $\mathbf{T}$. The latter implies that $x \in \mathbf{T}$ because $x$ is an index of the conjunction of the formulae such that their indices $k$ satisfy the above properties.
        
     \item  If $x$ belongs to $\mathbf{T}$, then, for each $k = \langle \overline{\varepsilon}, \sharp \beta, e^\prime \rangle$ such that $\beta < \alpha$ and such that $\UCF\left( e, k \right)$ converges in at most $t$ steps, it holds that $k \in \mathbf{T}$. The induction hypothesis implies that (\ref{eqn:2b''}) is derivable for $y = \mathtt{c}\left( e, t \right)$. If $y$ is not of this form, then (\ref{eqn:2b''}) is derivable too, hence (\ref{eqn:2b*}) is derivable. The derivability of (\ref{eqn:seq-main-bsc}) follows. %\qedhere
     \end{itemize}
\end{proof}

Theorem \ref{Th:lower-bound} now directly follows from Lemma \ref{lemm:lower-bound-main}.

\begin{proof}[Proof of Theorem \ref{Th:lower-bound}]
    The set $\mathbf{T}$ is many-one reducible to the set of sequents derivable from $\Hc$ in $\IAL$ by means of the function
    \[
    x\ \mapsto\
    \left(\VARlt, \VARsym_{1}^{x}, \VARsym_{\mathsf{Q}}, \VARen, \FOREn(k_M(x)), \ldots, \FOREn(k_1(x)), \FORBrk \yields \VARlt \mconj \VARok
    \right)
    \]
    where $\omega^{k_1(x)} + \ldots + \omega^{k_M(x)} = \alpha+1$ if $x = \langle \varepsilon, \sharp \alpha, e \rangle$. If $x$ is not of this form, then the reduction maps it to some underivable sequent, say, to $p \yields q$. Lemma \ref{lemm:lower-bound-main} implies that $x \in \mathbf{T}$ if and only if the sequent $${!}\Upsilon_{\Hc}, \VARlt, \VARsym_{1}^{x}, \VARsym_{\mathsf{Q}}, \VARen, \FOREn(k_M(x)), \ldots, \FOREn(k_1(x)), \FORBrk \yields \VARlt \mconj \VARok$$ is derivable in $\EIAL$. By Theorem \ref{Th:deduction} this is equivalent to the derivability of $\VARlt, \VARsym_{1}^{x}, \VARsym_{\mathsf{Q}}, \VARen, \FOREn(k_M(x)), \ldots, \FOREn(k_1(x)), \FORBrk \yields \VARlt \mconj \VARok$ from $\Hc$ in $\IAL$, which proves correctness of the reduction.
\end{proof}

Observe that the main construction uses only the connectives $\BS, \mconj, \aconj, {}^\ast$. Therefore, reasoning from monoidal inequations in the fragment of $\IAL$ where only $\BS, \mconj, \aconj, {}^\ast$ can be used is $\Sigma^0_{\omega^{\omega}}$-hard as well.

%%%

%\newpage
\section{Upper bound arguments}\label{S:Stanislav}

Now we are going to show that the closure ordinal of $\mathscr{D}^{-}$ is at most $\omega^{\omega}$, and the corresponding fixed point is reducible to $H \left( \omega^{\omega} \right)$. These complexity bounds turn out to be precise, as can be shown with the help of Theorem~\ref{Th:lower-bound}.

For our present purposes, it is useful to replace the finitary part of $\EIALminus$ by a special rule ($\mathrm{Fin}$) which operates as follows: $s$ is derivable from $s_1$, \ldots, $s_m$ by ($\mathrm{Fin}$) if{f} $s$ is derivable from $s_1$, \ldots, $s_m$ in $\EIALminus$ without using (${}^{\ast} L_{\omega}$). Define $\GEIALminus$ to be the resulting calculus, whose only rules are ($\mathrm{Fin}$) and the $\omega$-rule. Evidently, we have:

\begin{prop} \label{prop-gen}
Any sequent derivable in $\GEIALminus$ is derivable in $\EIALminus$, and vice versa.
\end{prop}

Naturally, $\underline{\mathscr{D}}^-$ will denote the immediate derivability operator of $\GEIALminus$. Moreover, if $s$ is a sequent derivable in $\GEIALminus$, we shall write $\underline{\mathrm{rank}}^{-} \left( s \right)$ for the rank of $s$ with respect to $\underline{\mathscr{D}}^-$.

\begin{theo} \label{theo-ubo}
The closure ordinals of $\underline{\mathscr{D}}^-$ and $\mathscr{D}^-$ are less than or equal to $\omega^{\omega}$.
\end{theo}

\begin{proof}
It is easy to verify that for every sequent $s$ derivable in $\EIALminus$,
\[
{\mathrm{rank}^{-} \left( s \right)}\ \leqslant\ {\omega \cdot \underline{\mathrm{rank}}^{-} \left( s \right)} .
\]
Hence it suffices to prove that each $\underline{\mathrm{rank}}^{-} \left( s \right)$, with $s$ as above, is less than $\omega^{\omega}$. To this end, we define the following operations on $\Ns$:
\begin{align*}
{\left( m_0, m_1, \ldots \right) \ncplus \left( n_0, n_1, \ldots \right)}\
&:=\ {\left( m_0 + n_0, m_1 + n_1, \ldots \right)} ;\\
{\left( m_0, m_1, \ldots \right) {\uparrow}}\
&:=\ {\left( 0, m_0, m_1, \ldots \right)} .
\end{align*}
Next, let $\mu$ be the function from the set of all formulae (not containing $^\ast$ in the scope of $!$) to $\Ns$ given by:
\begin{align*}
{\mu \left( A \right)}\
&:=\ {\left( 0, 0, 0, \ldots \right)} \quad \text{if}\ A\ \text{does~not contain}\ {}^\ast ;\\
{\mu \left( A \circ B \right)}\
&:=\ {\mu \left( A \right) \ncplus \mu \left( B \right)} \quad \text{for each} \enskip {\circ \in \left\{ \BS, \SL, \mconj, \aconj, \adisj \right\}} ;\\ %\text{if $A \circ B$ contains $\ast$}
{\mu \left( A^{\ast} \right)}\
&:= {\left( \mu \left( A \right) {\uparrow} \right) \ncplus \left( 2, 0, 0, \ldots \right)} .
\end{align*}
So in particular, if $A$ begins with $!$, then $\mu \left( A \right) = 0$. We extend $\mu$ to sequents by taking
\[
{\mu \left( A_1, \ldots, A_m \yields B \right)}\
:=\ {\mu \left( A_1 \right) \ncplus \ldots \ncplus \mu \left( A_n \right) \ncplus \mu \left( B \right)} .
\]
Denote $\mu \circ \nu$ by $\dot{\mu}$ (see Subsection~\ref{subsec-ha} for the definition of $\nu$). We claim that for every sequent $s$ derivable in $\GEIALminus$,
\begin{equation*}
{\underline{\mathrm{rank}}^{-} \left( s \right)}\ \leqslant\ {\dot{\mu} \left( s \right)} ,
\tag{$\dag$}
\end{equation*}
by induction on $\dot{\mu} \left( s \right)$. Evidently, if $\dot{\mu} \left( s \right) = 0$, then ${}^\ast$ does~not occur in $s$, and therefore $s$ can be derived by one application of ($\mathrm{Fin}$), which implies $\underline{\mathrm{rank}}^{-} \left( s \right) = 0$.\footnote{Obviously,
if $s$ is an {axiom}, we also have $\underline{\mathrm{rank}}^{-} \left( s \right) = 0$.}
Assume that $\dot{\mu} \left( s \right) > 0$, and consider an arbitrary derivation $\mathfrak{T}$ of $s$ in $\GEIALminus$.
\begin{enumerate}

\item[a.] Suppose that $\mathfrak{T}$ ends with an application of the $\omega$-rule, whose premises are $s_0$, $s_1$, \ldots; so $s$ is obtained from $s_0$, $s_1$, \ldots\ at the last step of  $\mathfrak{T}$. Then
\begin{align*}
{\underline{\mathrm{rank}}^{-} \left( s \right)}\
&\leqslant\
{\sup \left\{ \underline{\mathrm{rank}}^{-} \left( s_i \right) \mid i \in \omega \right\} + 1} \\ &\leqslant\
{\sup \left\{ \dot{\mu} \left( s_i \right) \mid i \in \omega \right\} + 1}\ <\
{\dot{\mu} \left( s \right)} ,
\end{align*}
where the second inequality is guaranteed by the inductive hypothesis.

\item[b.] Suppose that $\mathfrak{T}$ ends with an application of ($\mathrm{Fin}$). {Then} we can view $\mathfrak{T}$ as ending with a possibly broader application of ($\mathrm{Fin}$) with premises $s_1$, \ldots, $s_m$, each of which is either an axiom or derived by the $\omega$-rule in $\mathfrak{T}$. Observe that
\[
{\max \left\{ {\dot{\mu} \left( s_1 \right)}, \ldots, {\dot{\mu} \left( s_m \right)} \right\}}\ \leqslant\ {\dot{\mu} \left( s \right)}.
\]
Further, using the argument for (a), one can obtain $\underline{\mathrm{rank}}^{-} \left( s_i \right) < {\dot{\mu}} \left( s_i \right)$ for each $i = 1, \ldots, m$. Thus,
since ($\mathrm{Fin}$) contributes one more inference step, we have 
\[
{\underline{\mathrm{rank}}^{-} \left( s \right)}\ \leqslant\
{\max \left\{ {\underline{\mathrm{rank}}^{-} \left( s_1 \right)} , \ldots , {\underline{\mathrm{rank}}^{-} \left( s_m \right)} \right\} + 1}\ \leqslant\
{\dot{\mu} \left( s \right)} ,
\]
as desired.
\end{enumerate}
\end{proof}

To proceed further, we need a suitable technical lemma that will allow us to provide upper bounds for the iterations of any given {positive} elementary operator up to $\omega^{\omega}$ in a uniform way. Recall that each ordinal $\alpha$ can be uniquely represented as $\beta + n$ where $\beta \in \LOrd$ and $n \in \omega$. In particular, there are two functions $\mathrm{base}: \omega^{\omega} \rightarrow \LOrd$ and $\mathrm{step}: \omega^{\omega} \rightarrow \omega$ such that for every $\alpha \in \Ord$,
\[
\alpha\ =\
{\mathrm{base} \left( \alpha \right)}\ +\ {\mathrm{step} \left( \alpha \right)} .
\]
Moreover, if we identify the ordinals less than $\omega^{\omega}$ with their notations, then $\mathrm{base}$ and $\mathrm{step}$ --- as functions from $\omega$ to $\omega$ --- become computable, since
\begin{align*}
{\mathrm{base} \left( \rho \left( \left( m_0, m_1, m_2, \ldots \right) \right) \right)}\ &=\
{\rho \left( \left( 0, m_1, m_2, \ldots \right) \right)} ,\\
{\mathrm{step} \left( \rho \left( \left( m_0, m_1, m_2, \ldots \right) \right) \right)}\ &=\
{\rho \left( \left( m_0, 0, 0, \ldots \right) \right)} .
\end{align*}
Next, given a non-zero $n \in \omega$, define $\mathrm{lift}_n: \omega^{\omega} \rightarrow \omega^{\omega}$ by
\[
{\mathrm{lift}_n \left( \alpha \right)}\ :=\
{\mathrm{base} \left( \alpha \right) + \mathrm{step} \left( \alpha \right) \cdot n + 1},
\]
or, in our ordinal notation,
\[
{\mathrm{lift}_n \left( \rho \left( \left( m_0, m_1, m_2, \ldots \right) \right) \right)}\ =\ 
{\rho \left( \left( m_0 \cdot n + 1, m_1, m_2, \ldots \right) \right)} .
\]
Finally, to deal with many-one reductions, we shall use the numbering induced by the universal function $\mathtt{U}$. For any $S, T \subseteq \omega$, let
\[
{\mathrm{Index} \left( S; T \right)}\ :=\
{\left\{
k \in \omega \mid
\mathtt{U}_k ~ \text{many-one reduces} ~ S ~ \text{to} ~ T
\right\}} .
\]
So $k \in \mathrm{Index} \left( S; T \right)$ if{f} $\mathtt{U}_k$ is total and ${\left( \mathtt{U}_k \right)}^{-1} \left[ T \right] = S$. We call elements of $\mathrm{Index} \left( S; T \right)$ \emph{indices of $S$ with respect to $T$}. Each of these encodes a program that computes a function reducing $S$ to $T$, and therefore provides a way of showing that the complexity of $S$ is bounded by that of $T$. It is straightforward to establish the following; cf.\ \cite[Theorem~6.2]{Kuznetsov&Speranski-2023-SL}.\footnote{In
effect, the argument provided in \cite{Kuznetsov&Speranski-2023-SL} is more complicated, since it deals with Kleene's $\mathcal{O}$.}

\begin{lemm} \label{lemm-ubc}
Let $n \in \omega \setminus \left\{ 0 \right\}$ and $\Phi \left( x, X \right)$ be a positive $\Sigma^0_n$-for\-mu\-la. Then there is a computable $h: \omega \rightarrow \omega$ such that for every $\alpha < \omega^{\omega}$:
\begin{align*}
{h \left( \sharp \alpha \right)}\
&\in\
{\mathrm{Index} \left(
{{[ \Phi ]}^{\alpha} \left( \varnothing \right)},
{H \left( \mathrm{lift}_n \left( \alpha \right) \right)}
\right)}\\
&=\
{\mathrm{Index} \,(
{{[ \Phi ]}^{\alpha} \left( \varnothing \right)},
{\mathtt{J}^{\mathrm{step} \left( \alpha \right) \cdot n + 1} \left( H \left( \mathrm{base} \left( \alpha \right) \right) \right)}
)} .
\end{align*}
In particular, if $\alpha \in \LOrd$, then $h \left( \sharp \alpha \right)$ is an index of ${[ \Phi ]}^{\alpha} \left( \varnothing \right)$ with respect to $\mathtt{J} \left( H \left( \alpha \right) \right)$.
\end{lemm}

Combining Lemma~\ref{lemm-ubc} with Theorem~\ref{theo-ubo} we obtain:
  
\begin{theo}
The collection of all sequents derivable in $\EIALminus$ belongs to $\Sigma^0_{\omega^{\omega}}$.
\end{theo}

\begin{proof}
For convenience, take $S$ to be the collection in question, i.e.\ the least fixed point of $\mathscr{D}^{-}$. By Proposition~\ref{prop-gen}, $S$ coincides with the least fixed point of $\underline{\mathscr{D}}^{-}$. Clearly, we can find a positive $\Sigma^0_2$-for\-mu\-la $\Phi \left( x, X \right)$ such that $[ \Phi ] = \underline{\mathscr{D}}^{-}$. Let $\dot{\mu}$ be as in the proof of Theorem~\ref{theo-ubo}, and $h$ be as in the statement of Lemma~\ref{lemm-ubc}. Observe that for every sequent $s$,
\begin{align*}
s\ \in\ S \quad
&\overset{\ref{theo-ubo}}{\Longleftrightarrow} \quad
s\ \in\ {{\left( \underline{\mathscr{D}}^{-} \right)}^{\omega^{\omega}} \left( \varnothing \right)} \quad\\
&\Longleftrightarrow \quad
s\ \in\ {{\left( \underline{\mathscr{D}}^{-} \right)}^{\dot{\mu} \left( s \right) + 1} \left( \varnothing \right)}\\
&\overset{\ref{lemm-ubc}}{\Longleftrightarrow} \quad
{\mathtt{U}_{h \left( \sharp \left( \dot{\mu} \left( s \right) + 1 \right) \right)} \left( s \right)}\ \in\ {H \left( \mathrm{lift}_2 \left( \dot{\mu} \left( s \right) + 1 \right) \right)}\\[0.6em]
&\Longleftrightarrow \quad
{\mathtt{c} \left( {\mathtt{U}_{h \left( \sharp \left( \dot{\mu} \left( s \right) + 1 
 \right) \right)} \left( s \right)}, {\sharp \left( \mathrm{lift}_2 \left( \dot{\mu} \left( s \right) + 1 \right) \right)} \right)}\ \in\ {H \left( \omega^{\omega} \right)} ,
\end{align*}
where the second equivalence is justified by ($\dag$). Consequently, ${\left( \mathscr{D}^{-} \right)}^{\omega^{\omega}} \left( \varnothing \right) \leqslant H \left( \omega^{\omega} \right)$.
\end{proof}

This together with Theorem~\ref{Th:lower-bound} immediately gives:

\begin{coro}
The collection of all sequents derivable in $\EIALminus$ is $\Sigma^0_{\omega^{\omega}}$-comp\-le\-te, i.e.\ the least fixed point of $\mathscr{D}^{-}$ is many-one equivalent to $H \left( \omega^{\omega} \right)$.
\end{coro}

Turning to closure ordinals, here is one useful observation:

\begin{lemm} \label{lemm-lbo}
Let $\Phi \left( x, X \right)$ be a positive $\mathcal{L}_2$-for\-mu\-la with no set quantifiers. Suppose that $H \left( \omega^{\omega} \right)$ is many-one~re\-du\-cible to the least fixed point of $[ \Phi ]$. Then the closure ordinal of $[ \Phi ]$ is greater than or equal to $\omega^{\omega}$.
\end{lemm}

\begin{proof}
Take $\alpha$ to be the closure ordinal of $[ \Phi ]$; so ${[ \Phi ]}^{\alpha} \left( \varnothing \right)$ coincides with the least fixed point of $[ \Phi ]$. Suppose, for the sake of contradiction, that $\alpha < \omega^{\omega}$. Obviously, there exists a non-zero $n \in \omega$ such that $\Phi \left( x, X \right)$ is logically equivalent to some positive $\Sigma^0_n$-for\-mu\-la. Then
\[
{{[ \Phi ]}^{\alpha} \left( \varnothing \right)} \leqslant {H \left( \mathrm{lift}_n \left( \alpha \right) \right)}
\]
by Lemma~\ref{lemm-ubc}. Hence $H \left( \omega^{\omega} \right) \leqslant H \left( \mathrm{lift}_n \left( \alpha \right) \right)$, which contradicts Folklore~\ref{folk-hierarchy}.
\end{proof}

It follows that we cannot reach the closure of $\GEIALminus$ in less than $\omega^{\omega}$ steps:

\begin{coro} \label{coro-closure}
The closure ordinals of $\mathscr{D}^{-}$ and $\underline{\mathscr{D}}^{-}$ are both equal to $\omega^{\omega}$.
\end{coro}

\begin{proof}
Take $S$ to be the least fixed point of $\mathscr{D}^{-}$, or equivalently, that of $\underline{\mathscr{D}}^{-}$, by Pro\-po\-si\-ti\-on \ref{prop-gen}. Since $H \left( \omega^{\omega} \right) \leqslant S$ by Theorem~\ref{Th:lower-bound}, the closure ordinals of $\mathscr{D}^{-}$ and $\underline{\mathscr{D}}^{-}$ must be at least $\omega^{\omega}$ by Lemma~\ref{lemm-lbo}. On the other hand, Theorem~\ref{theo-ubo} ensures that they are~not greater than $\omega^{\omega}$.
\end{proof}

Of course, the same applies to the versions of $\mathscr{D}^{-}$ and $\underline{\mathscr{D}}^{-}$ that include the cut rule.

Let us briefly compare the results for $\EIALminus$ (and $\GEIALminus$) with those that are known for the initial system $\IAL$, without ${!}$ and hypotheses. The closure ordinal of the immediate derivability operator of the cut-free version of $\IAL$ is also exactly $\omega^\omega$; see \cite{Pshenitsyn2024MZ}. But it is not~known whether this is precise if we allow the cut rule.\footnote{In
particular, the sequents employed in the lower bound argument of \cite{Pshenitsyn2024MZ} can be reached much faster by using cut.}
Clearly, analogues of the robust argument of Corollary~\ref{coro-closure} (based on~Lemma~\ref{lemm-lbo}) do~not apply to $\IAL$, because the corresponding derivability problem --- unlike that for $\EIALminus$ --- is only $\Pi^0_1$-comp\-le\-te.

%%%

\section{Co-enumerability for a restricted fragment}\label{S:Stepan}

Recall that we consider the following fragments of $\EIALminus$: $\EIALm$, where ${!}$-for\-mu\-lae may be only of a very specific form, namely, ${!}((b_1 \cdot \ldots \cdot b_n) \BS (c_1 \cdot \ldots \cdot c_m))$, and an even more restrictive fragment $\EIALne$, where each such formula should satisfy the inequality $m \leqslant n$. In this section we shall show that this inequality makes a great difference. Namely, while $\EIALm$ is $\Sigma^0_{\omega^\omega}$-hard (Theorem~\ref{Th:lower-bound} above), for $\EIALne$ we shall establish a $\Pi^0_1$ upper complexity bound.

An informal sketch of the argument is as follows. In the previous section, all finite rules of $\EIALminus$ (i.e., all rules except the $\omega$-rule) were merged into one big rule called $(\mathrm{Fin})$, resulting in an auxiliary calculus $\GEIALminus$. Informally, applications of $(\mathrm{Fin})$ and $({}^* L_\omega)$ in a derivation correspond to $\exists$ and $\forall$ quantifiers, and they interleave up to the closure ordinal, which is $\omega^\omega$. This gives the $\Sigma^0_{\omega^\omega}$ upper bound. For the restricted fragment $\EIALne$, we shall show that its dyadic version $\dEIALne$, introduced in Section~\ref{S:dyadic}, has a {\em decidable} $(\mathrm{Fin})$ rule. Namely, for any sequent $s$ there will be a finite number of possible sets $\{ s_1, \ldots, s_m \}$ such that $s$ is derivable from $s_1, \ldots, s_m$ without using $({}^* L_\omega)_{\mathrm{d}}$. Thus, the only rule which increases the complexity will be $({}^* L_\omega)_{\mathrm{d}}$, and we shall show that $\omega^\omega$ iterations of that still result in $\Pi^0_1$.

A similar argument was used in~\cite{KuznetsovSperanski2022} to establish a $\Pi^0_1$ upper bound for $\IAL$ with a family of subexponentials which could allow weakening and permutation, but not contraction. Moreover, we shall use the core statement~\cite[Proposition 5.14]{KuznetsovSperanski2022} ``out-of-the-box.'' In the case of $\EIALne$, however, this argument cannot be applied to the calculus in its original formulation. Namely, the rule $(\mathrm{Fin})$ obtained from $\EIALne$ does not satisfy the decidability property formulated above. Indeed, if the antecedent of $s$ contains a ${!}$-for\-mu\-la, then one may apply $({!}C)$ any number of times, yielding an infinite number of ways to obtain $s$ by application of $(\mathrm{Fin})$. In order to resolve this issue, we use the dyadic system $\dEIALne$.

By $(\mathrm{Fin})_{\mathrm{d}}$ let us denote the version of the $(\mathrm{Fin})$ rule for the dyadic system $\dEIALne$: $s$ is derivable from $s_1, \ldots, s_m$ by $(\mathrm{Fin})_{\mathrm{d}}$ if{f} $s$ is derivable from $s_1, \ldots, s_m$ in $\dEIALne$ without using $({}^*L_\omega)_{\mathrm{d}}$. Notice that such a rule could also be formulated for $\dEIALm$, but it will not have the necessary decidability properties. The calculus with two rules, $(\mathrm{Fin})_{\mathrm{d}}$ and $({}^*L_\omega)_{\mathrm{d}}$, is denoted by $\GdEIALne$, and it is obviously equivalent to $\dEIALne$, and therefore to $\EIALne$.

The function $\mu$ (see proof of Theorem~\ref{theo-ubo}) is extended to d-sequents in the following way:
\[
{\mu({!}\Xi; A_1, \ldots, A_n \yields B)}\ :=\
{\mu(A_1) \ncplus \ldots \ncplus \mu(A_n) \ncplus \mu(B)} .
\]
(It ignores the $!$-zone, since the ${!}$-zone does not include ${}^*$.)
By $\underline{\mathrm{rank}}_{\mathrm{d}}^{\mathrm{ne}}(s)$ we denote the rank of the sequent $s$ w.r.t.\ the immediate derivability operator of $\GdEIALne$ (in its cut-free version). Naturally, the function $\mu$ and rank are connected by the following inequation:
\begin{equation*}
{\underline{\mathrm{rank}}_{\mathrm{d}}^{\mathrm{ne}}(s)}\ \leqslant\ \dot{\mu}(s).
\tag{$\dag_{\mathrm{d}}$}
\end{equation*}
The argument is the same as that for Theorem~\ref{theo-ubo}, since the rules operating ${}^*$ are basically the same. In particular, the closure ordinal for $\GdEIALne$ is also bounded by $\omega^\omega$.

Now, following~\cite{KuznetsovSperanski2022}, we fix effective G\"{o}del  numberings for:
\begin{itemize}
\item d-sequents, denoted by $\#(s)$;
\item finite sequences of d-sequents, denoted by $\#_*(s_1, \ldots, s_m)$;
\item infinite sequences of d-sequents of the form
\[
s_0 = {!}\Xi; \Gamma, \Delta \yields C, \quad
s_1 = {!}\Xi; \Gamma, A, \Delta \yields C,
\quad s_2 = {!}\Xi; \Gamma, A^2, \Delta \yields C, \quad \ldots,
\]
denoted by $\#_{\bullet}(s_0, s_1, s_2, \ldots)$. Such infinite sequences are sequences of premises for $({}^*L)_\mathrm{d}$.
\end{itemize}
Next, let us define two mappings:
\begin{align*}
{R_*(\# s)}\ &:=\
{\left\{ {\#_*(s_1, \ldots, s_m)} \mid \mbox{$s$ is obtained from $s_1, \ldots, s_m$ by $(\mathrm{Fin})_{\mathrm{d}}$} \right\}} , \\
{R_{\bullet}(\#s)}\ &:=\
{\left\{ {\#_{\bullet} (s_0, s_1, s_2, \ldots)} \mid \mbox{$s$ is obtained from $s_0, s_1, s_2, \ldots$ by  $({}^* L_\omega)_{\mathrm{d}}$} \right\}} . 
\end{align*}

\begin{lemm}
For each d-sequent $s$, the sets $R_*(\#s)$ 
and $R_{\bullet}(\#s)$ 
are finite and computable uniformly in $\#s$.
\end{lemm}

\begin{proof}
For $R_{\bullet}(\#s)$, the statement is evident. Let us prove the statement for $R_{*}(\#s)$.

Recall that the complexity of a formula~$A$, denoted by $|A|$, is defined as the total number of variable, constant, and connective occurrences in $A$.
For each d-sequent we define a parameter $c = c({!}\Xi; \Gamma \yields C)$ as $|\Gamma| + |C|$, where $|A|$ is the complexity of $A$ and $|\Gamma|$ is the sum of complexities of formulae in $\Gamma$. Notice that $!$-for\-mu\-lae in ${!}\Xi$ do not count.
For each rule of $\dEIALne$, except $({}^* L_\omega)_{\mathrm{d}}$, the $c$ parameter for each of its premises is not greater than the $c$ parameter of the conclusion. The crucial case here is the $(A)_{\mathrm{d}}$ rule. Recall its formulation:
\[
\infer[(A)_{\mathrm{d}}]
{\{{!}((b_1 \cdot \ldots \cdot b_n) \BS (c_1 \cdot \ldots \cdot c_m)) \} \cup {!}\Xi; \Gamma, b_1, \ldots, b_n, \Delta \yields C}
{\{{!}((b_1 \cdot \ldots \cdot b_n) \BS (c_1 \cdot \ldots \cdot c_m)) \} \cup  {!}\Xi; \Gamma, c_1, \ldots, c_m, \Delta \yields C} 
\]
Here for ${!}((b_1 \cdot \ldots \cdot b_n) \BS (c_1 \cdot \ldots \cdot c_m))$ we have $m \leqslant n$, which yields exactly the necessary inequation 
$|c_1, \ldots, c_m| \leqslant |b_1, \ldots, b_n|$. For the system $\dEIALm$, without the $m \leqslant n$ restriction, this argument would have failed.

Next, we use subformula properties of a cut-free derivation in $\dEIALne$. First, in such a derivation, all formulae occurring in ${!}$-zones are subformulae (actually, ${!}$-subformulae) of the goal d-sequent. Second, each variable inside such a derivation should also occur in the goal sequent. The part outside the $!$-zone is also bounded by the $c$ parameter, thus, there is only a finite  choice of d-sequents which may occur in the derivation. Moreover, this choice is computably bounded w.r.t.\ the goal sequent $s$.

Thus, the set $R_*(\#s)$ is finite. The algorithm for computing it from $\#s$ is exhaustive proof search. Here we use the fact that the height of proof tree is bounded by the size of $R_*(\#s)$ (which is computable): if a d-sequent appears in a derivation branch twice, we may abandon this line of proof search.
\end{proof}

\begin{prop}\label{prop-pi01}
There exists a computable function $g$ from $\mathscr{N}$ to $\omega$ such that, given $\alpha \in \mathscr{N}$, it returns the G\"{o}del number of a $\Pi^0_1$-for\-mu\-la defining the set
\[
\{ \#s \mid \dot{\mu}(s) = \alpha \mbox{ and $s$ is derivable in $\GdEIALne$ } \}
\]
in the standard model $\mathfrak{N}$ of arithmetic.
\end{prop}

\begin{proof}
Exactly as that of~\cite[Proposition~5.14]{KuznetsovSperanski2022}.
\end{proof}

\begin{theo}\label{theo-pi01}
The collection of all sequents derivable in $\EIALne$ is $\Pi^0_1$-comp\-le\-te.
\end{theo}

\begin{proof}
Since $\EIALne$ is equivalent to $\GdEIALne$, we shall prove the upper bound for the latter. The argument comes from~\cite[Corollary 5.15]{KuznetsovSperanski2022}. Let $g$ be the function given by Proposition~\ref{prop-pi01}. We can define the set of G\"odel numbers of all d-sequents derivable in $\GdEIALne$ by $\Pi^0_1\mbox{-SAT}(x, g(\dot{\mu}(x)))$, where $\Pi^0_1\mbox{-SAT}$ is the standard $\Pi^0_1$-formula expressing satisfiability for $\Pi^0_1$-formulae with exactly one free variable. 

The complexity lower bound follows from that for $\IAL$, i.e.\ infinitary action logic without exponentials, proved by Buszkowski~\cite{Buszkowski2007}.
\end{proof}

\begin{coro}
The following problem is $\Pi^0_1$-comp\-le\-te: given a finite set $S$ of commutativity conditions and an equation $A = B$ in the language of action lattices, decide whether $S$ logically implies $A = B$ in all $*$-con\-ti\-nu\-ous action lattices.
\end{coro}

\begin{proof}
The upper bound follows from Theorem~\ref{theo-pi01} via Theorem~\ref{Th:deduction}. The lower bound follows from that for reasoning from commutativity conditions on $*$-con\-ti\-nu\-ous Kleene algebras; see~\cite{Kozen2002}.
\end{proof}

Thus we have shown that reasoning from commutativity conditions in $*$-con\-ti\-nu\-ous action lattices is~not harder than the corresponding problem for $*$-con\-ti\-nu\-ous Kleene algebras.

%[ Stepan Kuznetsov ]

%%%

\subsection*{Funding} %Acknowledgements}

This work was supported by the Russian Science Foundation under grant no.~23-11-00104, \href{https://rscf.ru/en/project/23-11-00104/}{https://rscf.ru/en/project/23-11-00104/}.

%%%

%\newpage
\bibliographystyle{asl}
\bibliography{act}

\addcontentsline{toc}{section}{References} %!!!

%%%

%%%

%\vspace{2em}
%\noindent
%{\small
%{\scshape Stepan L.\ Kuznetsov, \enskip Tikhon \ Pshenitsyn, %\enskip Stanislav O.\ Speranski}\\[0.5em]
%Steklov Mathematical Institute of Russian Academy of Sciences\\
%8 Gubkina St., 119991 Moscow, Russia\\[0.5em]
%{\ttfamily sk@mi-ras.ru}, \enskip {\ttfamily tpshenitsyn@mi-%ras.ru}, \enskip {\ttfamily katze.tail@gmail.com}
%}

%%%

%\newpage
%\appendix
\subsection*{Appendix: Proof of Proposition 4.3}
%\label{appendix}

%

%\addcontentsline{toc}{section}{Appendix} %Proof of Proposition 4.3 %!!!

%

Let us prove Proposition \ref{prop:hardness-T}. Given an index $i = \langle \varepsilon, \sharp \alpha, e \rangle$, let $\ips(i)$ denote the infinitary propositional sentence corresponding to $i$. Then, Definition \ref{Df:infinitary_propositional_sentences} can be represented as follows for each non-zero $\alpha \in \omega^{\omega}$:
\begin{align*}
{\ips(\langle 0, \sharp \alpha, e \rangle)}\
&=\
\bigdoublevee\limits_{\substack{\beta < \alpha \\ \langle 1, \sharp \beta, e^\prime \rangle \in \mathtt{W}_e}}
\ips(\langle 1, \sharp \beta, e^\prime \rangle) ;\\
{\ips(\langle 1, \sharp \alpha, e \rangle)}\
&=\
\bigdoublewedge\limits_{\substack{\beta < \alpha \\ \langle 1, \sharp \beta, e^\prime \rangle \in \mathtt{W}_{e}}}
\ips(\langle 0, \sharp \beta, e^\prime \rangle).
\end{align*}
Here $\mathtt{W}_e$ denotes $\dom \UCF_{e}$. Besides, $\ips(\langle \varepsilon, \sharp 0, 0 \rangle) = \bot$ and $\ips(\langle \varepsilon, \sharp 0, e \rangle) = \top$ where $\varepsilon \in \{0,1\}$ and $e > 0$.
We write $\omega_+$ for $\omega \setminus \{ 0 \}$. Note that $\omega_+ = \left\{ \sharp \alpha \mid \alpha \in \omega^\omega\right\}$. 

A computable function $f: \{\sharp \beta \mid \beta \le \alpha\} \to \omega$ is called {\em $\alpha$-app\-ro\-pria\-te} if, for each $\beta \le \alpha$:
    \begin{enumerate}
        \item $\UCF_{f\left(\sharp \beta\right)}$ is a total function;
        \item\label{item-appropriate-2} for each $m \in \omega$, $\UCF_{f\left(\sharp \beta\right)}(m)$ is a $\Sigma_{\alpha}$-index and belongs to $\mathbf{T}$ if and only if $\mathtt{c} \left( m, {\sharp \alpha} \right)$ belongs to $H \left( \omega^\omega \right)$ (equivalently, $m \in H \left( \alpha \right)$).
    \end{enumerate}
We shall define a computable function $f:\omega_+ \to \omega$ such that, for each $\alpha < \omega^\omega$, the restriction of $f$ to $\{\sharp \beta \mid \beta \le \alpha\}$ is $\alpha$-app\-ro\-pria\-te. Then, clearly, using it, one can reduce $H \left( \omega^{\omega} \right)$ to $\mathbf{T}$. Such a function $f$ will be constructed using effective transfinite recursion \cite[Section I.4.1]{Montalban2022}. According to this method, for any computable function $g$ of two arguments, there exists $e \in \omega$ such that $\UCF_{e}\left( \sharp \alpha \right) = g(\sharp \alpha ,e{\restriction_{<\alpha}})$. Here $e{\restriction_{<\alpha}}$ is the index of the following partial function:
    \[
    {\UCF (e{\restriction_{<\alpha}},n)}\
    =\
    \begin{cases}
        \UCF(e,\sharp \beta) & {n = {\sharp \beta}} \text{ and } {\beta<\alpha}\\
        \text{undefined} &\text{otherwise} .
    \end{cases}
    \]
    (Note that there are many indices corresponding to the same function. Here we assume that the index $e{\restriction_{<\alpha}}$ is obtained from $e$ and $\sharp \alpha$ in a fixed way by straightforwardly altering the Turing machine defined by $e$.)
    
    Let us define $g$ in such a way that the function $f := \UCF_{e}$ obtained from it by effective transfinite recursion is a desirable one. First, let $g\left(\sharp 0, l \right)$ be an index of the following constant function: $\UCF\left(g\left(\sharp 0, l \right),m \right) = \langle 0, \sharp 0, 0 \rangle$. This is a $\Sigma_0$-index corresponding to the sentence $\bot$. Since $H \left( 0 \right) = \varnothing$, this agrees with the property \ref{item-appropriate-2} of $0$-app\-ro\-pria\-te functions.

    Secondly, let $\alpha = \beta+1$. Recall that
    \begin{align*}
    {\mathtt{c}\left(m,\sharp \alpha \right)}\ \in\ {H \left( \omega^{\omega} \right)}
    \enskip &\Longleftrightarrow \enskip
    m\ \in\ {H \left( \alpha \right) = \mathtt{J} \left( H \left( \beta \right) \right)}
    \\ &\Longleftrightarrow \enskip
    {\UCF^{H \left( \beta \right)} (m,m)} \enskip \text{converges} .
    \end{align*}
    If $\UCF^{H \left( \beta \right)}(m,m)$ converges, then there exists $s \in \omega$ such that the computation of $\UCF^{H \left( \beta \right)}(m,m)$ halts after at most $s$ steps and such that, during the computation, the oracle says whether $i \in H \left( \beta \right)$ only for $i < s$. Let us say that, if $\sigma = (\sigma(0),\ldots, \sigma(s-1)) \in 2^{<\omega}$ is a finite sequence representing the initial part of an oracle, then $\UCF^{\sigma}(e,n)$ converges if and only if, for $S = \left\{ i<s \mid \sigma(i) = 1\right\}$, the computation of $\UCF^{S}(e,n)$ halts after at most $s$ steps and the oracle is asked whether $i \in H \left( \beta \right)$ only if $i < s$. This fact is denoted by $\UCF^{\sigma}(e,n) \downarrow$.

    To define $g\left( \sharp \alpha, l \right)$ imagine that $l$ is an index of a $\beta$-app\-ro\-pria\-te function $\UCF_l$. Then, for each $m \in \omega$, the sentence $\ips(\UCF\left( g\left( \sharp \alpha , l \right), m \right))$ must be logically equivalent to the following one:
    \begin{equation*}\label{eqn:appropriate-function-successor}
    \bigdoublevee\limits_{\substack{\sigma \in 2^{<\omega} \\ \UCF^{\sigma}(m,m)\downarrow}}
    \left(
    \bigdoublewedge\limits_{\substack{i < \vert \sigma \vert \\ \sigma(i) = 1}}
    \ips(\UCF\left( \UCF\left(l,\sharp \beta\right), i \right))
    \wedge
    \bigdoublewedge\limits_{\substack{i < \vert \sigma \vert \\ \sigma(i) = 0}}
    \neg\ips(\UCF\left( \UCF\left(l,\sharp \beta\right), i \right))
    \right) . \tag{$\ddag$}
    \end{equation*}
    According to the definition of a $\beta$-app\-ro\-pria\-te function, $\UCF\left( \UCF\left(l,\sharp \beta\right), i \right)$ is a $\Sigma_\beta$-index which belongs to $\mathbf{T}$ if and only if $i \in H \left( \beta \right)$. Therefore, the sentence (\ref{eqn:appropriate-function-successor}) is true if and only if there exists a finite sequence $\sigma$ such that $\UCF^{\sigma}(m,m)\downarrow$ and such that $\sigma(i)=1$ if and only if $i \in H \left( \beta \right)$. This is equivalent to the fact that $\UCF^{H \left( \beta \right)}(m,m)$ converges and thus to the fact that $\mathtt{c}\left(m,\sharp \alpha \right) \in H \left( \omega^{\omega} \right)$.

    It remains to explain how to define $g\left( \sharp \alpha , l \right)$ in a computable way so that $\UCF\left( g\left( \sharp \alpha , l \right), m \right)$ is a $\Sigma_\alpha$-index. The argument exploits the following basic facts about computable infinitary sentences (most of them are proved by effective transfinite recursion).
    \begin{enumerate}
        \item\label{item-cif-basic-1} Given a $\Sigma_\beta$-index of an infinitary propositional sentence, one can effectively find a $\Pi_\beta$-index of a sentence logically equivalent to its negation.
        \item\label{item-cif-basic-2} Given a finite set of $\Sigma_\beta$-indices/$\Pi_\beta$-indices of sentences, one can effectively find a $\Sigma_\beta$-index/$\Pi_\beta$-index of a sentence logically equivalent to the conjunction of these sentences.
        \item\label{item-cif-basic-3} Given a $\Sigma_\beta$-index and a $\Pi_\beta$-index, one can effectively find the $\Sigma_{\beta+1}$-index of a sentence logically equivalent to the conjunction of the sentences corresponding to these two indices.
        \item\label{item-cif-basic-4} One can effectively find the $\Sigma_\alpha$-index of a sentence logically equivalent to computably enumerable disjunction of sentences with $\Sigma_\alpha$-indices.
        \item\label{item-cif-basic-5} Given a $\Sigma_\beta$-index, one can effectively transform it into a $\Sigma_\alpha$-index for each $\alpha > \beta$.
    \end{enumerate}
    
     Therefore, according to the properties \ref{item-cif-basic-1}, \ref{item-cif-basic-2} and \ref{item-cif-basic-3}, given $l$ and $\sharp \beta$, one can compute the $\Sigma_\alpha$-index of a sentence logically equivalent to
    \[
        \bigdoublewedge\limits_{\substack{i < \vert \sigma \vert \\ \sigma(i) = 1}}
    \ips(\UCF\left( \UCF\left(l,\sharp \beta\right), i \right))
    \wedge
    \bigdoublewedge\limits_{\substack{i < \vert \sigma \vert \\ \sigma(i) = 0}}
    \neg\ips(\UCF\left( \UCF\left(l,\sharp \beta\right), i \right)),
    \]
    and then, using property \ref{item-cif-basic-4}, effectively find a $\Sigma_\alpha$-index of a sentence logically equivalent to (\ref{eqn:appropriate-function-successor}).

    Finally, let $\alpha$ be limit and let $l$ be an index of a function which is $\beta$-app\-ro\-pria\-te for each $\beta < \alpha$. For each $m = \mathtt{c}\left(k,\sharp \alpha^\prime\right)$ where $k \in \omega$ and $\alpha^\prime < \alpha$, let $\UCF\left( g\left( \sharp \alpha , l \right), m \right)$ be a $\Sigma_\alpha$-index of a sentence logically equivalent to the sentence $\ips\left(\UCF\left(\UCF\left( l, \sharp \alpha^\prime \right), k \right)\right)$ (here the property \ref{item-cif-basic-5} is used). Otherwise, let $\UCF\left( g\left( \sharp \alpha , l \right), m \right)$ be the $\Sigma_\alpha$-index of the sentence $\bot$. Then, indeed, $\ips \left(\UCF\left( g\left( \sharp \alpha , l \right), m \right)\right)$ is true if and only if $m$ is of the form $m = \mathtt{c}\left(k,\sharp \alpha^\prime\right)$ where $k \in \omega$, $\alpha^\prime < \alpha$, and $k \in H \left( \alpha^\prime \right)$; the latter is equivalent to $m \in H\left(\alpha\right)$, which is, in turn, equivalent to the fact that $\mathtt{c}(m,\sharp \alpha) \in H \left(\omega^\omega \right)$.

%%%

\end{document}